\documentclass[10pt,a4paper]{article}
    \usepackage{amsmath,,amssymb,amsthm,mathtools,mathrsfs}
    \usepackage{enumerate,enumitem}
    \usepackage[all]{xy}
    \usepackage{tikz-cd}
    \usepackage[margin=25mm]{geometry}
    \usepackage{titlesec}
        \titleformat{\section}{\normalfont\large\bf}{\thesection.}{1ex}{\centering}
        \titleformat{\subsection}[runin]{\normalfont\bf}{\thesubsection.}{1ex}{}[.]
    \theoremstyle{plain}
        \newtheorem{theorem}{Theorem}[section]
        \newtheorem{lemma}[theorem]{Lemma}
        \newtheorem{proposition}[theorem]{Proposition}
        \newtheorem{corollary}[theorem]{Corollary}
    \theoremstyle{definition}
        
    \theoremstyle{remark}
        \newtheorem{remark}[theorem]{Remark}
    \DeclareMathOperator{\supp}{supp}
    \numberwithin{equation}{section}
\begin{document}
\begin{center}\large
    {\bf Properties of Navier-Stokes mild solutions in sub-critical Besov \\ spaces  whose regularity exceeds the critical value by $\boldsymbol{\epsilon\in(0,1)}$}\\\ \\
    Joseph P.\ Davies\footnote{University of Sussex, Brighton, UK, {\em jd535@sussex.ac.uk}} and Gabriel S.\ Koch\footnote{University of Sussex, Brighton, UK, {\em
    g.koch@sussex.ac.uk}}
\end{center}
\abstract{
We consider mild solutions to the Navier-Stokes initial-value problem which belong to certain ranges $Z_{p,q}^{s}(T,n):=\widetilde{L}^{1}(0,T;\dot{B}_{p,q}^{s+2}(\mathbb{R}^{n}))\cap\widetilde{L}^{\infty}(0,T;\dot{B}_{p,q}^{s}(\mathbb{R}^{n}))$ of Chemin-Lerner spaces. For $\epsilon\in(0,1)$, $s\in(-1,\infty)$, $p,q\in[1,\infty]$, and initial data $f\in\dot{B}_{p,q}^{s}(\mathbb{R}^{n})\cap\dot{B}_{\infty,\infty}^{-1+\epsilon}(\mathbb{R}^{n})$, we prove that there exists, at least locally in time, a unique solution $u\in\cap_{T'\in(0,T)}\left(Z_{p,q}^{s}(T',n)\cap Z_{\infty,\infty}^{-1+\epsilon}(T',n)\right)$. The solution $u$ and its maximal existence time $T_{f}^{*}\in(0,\infty]$ are independent of $\epsilon,s,p,q$. If $T_{f}^{*}$ is finite, then we have the blow-up estimate ${\|u(t)\|}_{\dot{B}_{\infty,\infty}^{-1+\epsilon}(\mathbb{R}^{n})}\gtrsim_{\varphi}\epsilon(1-\epsilon){(T_{f}^{*}-t)}^{-\epsilon/2}$ for all $t\in(0,T_{f}^{*})$, where $\varphi$ is the cutoff function used to define the Littlewood-Paley projections $\dot{\Delta}_{j}:=\varphi(2^{-j}D)$. The solution is unique among all solutions in the larger class $\cap_{T'\in(0,T_{f}^{*})}\cup_{\alpha\in(2,\infty)}L^{\alpha}(0,T';L^{\infty}(\mathbb{R}^{n}))$, and if $T_{f}^{*}<\infty$ then $u\notin L^{2}(0,T_{f}^{*};L^{\infty}(\mathbb{R}^{n}))$.

For $n=3$, $\epsilon\in(0,1)$ and $f\in\dot{B}_{\infty,\infty}^{-1+\epsilon}(\mathbb{R}^{n})$, we note that Chemin and Gallagher (Tunis. J. \linebreak Math., 2019) construct a solution $u\in\cap_{T'\in(0,T_{f,\epsilon}^{*})}Z_{\infty,\infty}^{-1+\epsilon}(T',n)$ with maximal existence time 
\linebreak 
${T_{f,\epsilon}^{*}\gtrsim_{\varphi,\epsilon}{\|f\|}_{\dot{B}_{\infty,\infty}^{-1+\epsilon}(\mathbb{R}^{n})}^{-2/\epsilon}}$. The results announced above are an improvement on the result by Chemin and Gallagher: we establish certain `propagation of regularity' results to show that the solution and its existence time are indepenent of $\epsilon$, we obtain a more explicit dependence on $\epsilon$ in the implied constants, and we establish additional properties of the solution, depending on the Besov spaces to which the initial data belongs.

We also address measurability issues which receive limited explicit attention in the existing literature.
}
\section{Introduction}
We consider the ``mild'' (or fixed-point) formulation of the Navier-Stokes initial value problem: for $T\in(0,\infty]$ and a tempered distribution $f\in\mathcal{S}'(\mathbb{R}^{n};\mathbb{R}^{n})$, we are interested in weakly* measurable functions\footnote{A function $u:(0,T)\rightarrow\mathcal{D}'(\mathbb{R}^{n})$ is said to be weakly* measurable if the map $t\mapsto\langle u(t),\phi\rangle$ is measurable for all $\phi\in\mathcal{D}(\mathbb{R}^{n})$. We discuss weak* measurability in Section \ref{weak*-measurability} below.} $u:(0,T)\rightarrow\mathcal{S}'(\mathbb{R}^{n};\mathbb{R}^{n})$ which satisfy
\begin{equation*}
    u(t) = S[f](t) - B[u,u](t) \quad \text{in }\mathcal{S}'(\mathbb{R}^{n})\text{ for all }t\in(0,T),
\end{equation*}
where the operators $S$ and $B$ are given by\footnote{We use the notation $\langle v,\phi\rangle$ to denote the action of a distribution $v$ on a test function $\phi$. We also use the notation $\langle \psi,\phi\rangle$ to denote the integral $\int_{\mathbb{R}^{n}}\psi(x)\phi(x)\,\mathrm{d}x$ if $\psi$ and $\phi$ are functions and $\psi\phi$ is integrable. If $v\in\mathcal{D}'(\mathbb{R}^{n})$ and $\psi\in L_{\mathrm{loc}}^{1}(\mathbb{R}^{n})$ satisfy $\langle v,\phi\rangle\equiv\langle\psi,\phi\rangle$, then we use the abuse of notation $v=\psi$.}
\begin{equation*}
    S[f](t) := e^{t\Delta}f,
\end{equation*}
\begin{equation*}
    B[u,v] := G\mathbb{P}\nabla\cdot(u\otimes v),
\end{equation*}
\begin{equation*}
    \langle{(\mathbb{P}\nabla\cdot W)}_{i},\phi\rangle := -\langle W_{jk},\mathbb{P}_{ij}\nabla_{k}\phi\rangle \quad \text{for }\phi\in\mathcal{S}(\mathbb{R}^{n}),
\end{equation*}
\begin{equation*}
    \langle G[w](t),\phi\rangle := \int_{0}^{t}\langle e^{(t-s)\Delta}w(s),\phi\rangle\,\mathrm{d}s \quad \text{for }\phi\in\mathcal{S}(\mathbb{R}^{n}),
\end{equation*}
and $\mathbb{P}_{ij}=\delta_{ij}-\frac{\nabla_{i}\nabla_{j}}{\Delta}$ is the projection onto divergence-free vector fields. In some contexts where the pointwise product $u\otimes v$ or the operator $\mathbb{P}\nabla\cdot$ are not necessarily defined, we replace the bilinear operator $B$ by the operator
\begin{equation*}
    \widetilde{B}[u,v] := G\left(\sum_{j\in\mathbb{Z}}\mathbb{P}\nabla\cdot\dot{\Delta}_{j}\mathsf{Bony}(u,v)\right),
\end{equation*}
where the Littlewood-Paley projections $\dot{\Delta}_{j}:=\varphi(2^{-j}D)$ (for a certain cutoff function $\varphi$) and the Bony product $\mathsf{Bony}(u,v)$ are defined in Section \ref{besov-section} below. If $f$ is divergence-free, then $u$ is too, but we will not use divergence-freedom of $f$ or $u$ at any point in this paper.

We consider the fixed-point problem in the context of the (semi)norms
\begin{equation*}
    {\|u\|}_{L_{T}^{\alpha}L^{p}(\mathbb{R}^{n})} := {\left\|t\mapsto{\|u(t)\|}_{L^{p}(\mathbb{R}^{n})}\right\|}_{L^{\alpha}(0,T)} \quad \text{for }\alpha,p\in[1,\infty]\text{ and }T\in(0,\infty],
\end{equation*}
\begin{equation*}
    {\|u\|}_{\widetilde{L}_{T}^{\alpha}\dot{B}_{p,q}^{s}(\mathbb{R}^{n})} := {\left\|j\mapsto2^{js}{\|\dot{\Delta}_{j}u\|}_{L_{T}^{\alpha}L^{p}(\mathbb{R}^{n})}\right\|}_{l^{q}(\mathbb{Z})} \quad \text{for }\alpha,p,q\in[1,\infty],\text{ }s\in\mathbb{R}\text{ and }T\in(0,\infty].
\end{equation*}
We use the notation $L_{T}^{\overline{\infty}}$ or $\widetilde{L}_{T}^{\overline{\infty}}$ when we want to replace ${\|\cdot\|}_{L^{\infty}(0,T)}$ by the supremum norm. The spaces $\widetilde{L}_{T}^{\alpha}\dot{B}_{p,q}^{s}(\mathbb{R}^{n})$ were first introduced in \cite{chemin1995} (in the case of the space $\widetilde{L}_{T}^{1}\dot{B}_{2,2}^{(n+2)/2}(\mathbb{R}^{n})$), and satisfy the property that if $f$ belongs to the Besov space $\dot{B}_{p,q}^{s}(\mathbb{R}^{n})$, then for all $T\in(0,\infty]$ we have that $S[f]$ belongs to\footnote{The inclusion $\subseteq$ stated here follows from the embeddings \eqref{interpolation-holder} and \eqref{chemin-lerner-minkowski} given below.}
\begin{equation*}
    Z_{p,q}^{s}(T) = Z_{p,q}^{s}(T,n) := \widetilde{L}_{T}^{1}\dot{B}_{p,q}^{s+2}(\mathbb{R}^{n})\cap\widetilde{L}_{T}^{\overline{\infty}}\dot{B}_{p,q}^{s}(\mathbb{R}^{n}) \subseteq L_{T}^{\overline{\infty}}\dot{B}_{p,q}^{s}(\mathbb{R}^{n})\cap\bigcap_{\alpha\in[1,\infty]}\widetilde{L}_{T}^{\alpha}\dot{B}_{p,q}^{s+\frac{2}{\alpha}}(\mathbb{R}^{n}).
\end{equation*}

We can make sense of the fixed-point problem in classes which are more general than $Z_{p,q}^{s}$. For $T\in(0,\infty)$, and $(\alpha,\ell,\epsilon)\in{[1,\infty]}^{2}\times(0,1]$ satisfying\footnote{The inequality $-s_{2}<\epsilon+\frac{2}{\alpha}$ is automatically satisfied in the case $n\geq2$. We choose to not exclude the case $n=1$, because the results of this paper hold not only for the Navier-Stokes equation $\partial_{t}u-\Delta u = -\mathbb{P}\nabla\cdot(u\otimes u)$, but also for any heat equation of the form $\partial_{t}u-\Delta u = \sigma(D)[uu]$, where $\sigma(D)$ is given by a smooth Fourier multiplier which is postive homogeneous of degree 1.}
\begin{equation}\label{ale-conditions}
    \left\{\begin{array}{l} \text{either } s_{\ell}+\epsilon+\frac{2}{\alpha}=0 \text{ and } -s_{2}\leq\epsilon+\frac{2}{\alpha}\leq1, \\ \text{or } s_{\ell}+\epsilon+\frac{2}{\alpha}>0 \text{ and } -s_{2}<\epsilon+\frac{2}{\alpha}<1, \end{array}\right.
\end{equation}
where $s_{\ell}=s_{\ell}(n):=-1+\frac{n}{\ell}$, we consider the path space
\begin{equation*}
    X_{\ell,\epsilon}^{\alpha}(T) = X_{\ell,\epsilon}^{\alpha}(T,n) := \left\{\begin{array}{ll} L_{T}^{\alpha}L^{\ell}(\mathbb{R}^{n}) & \text{if }s_{\ell}+\epsilon+\frac{2}{\alpha}=0, \\ \widetilde{L}_{T}^{\alpha}\dot{B}_{\ell,\infty}^{s_{\ell}+\epsilon+\frac{2}{\alpha}}(\mathbb{R}^{n}) & \text{if } s_{\ell}+\epsilon+\frac{2}{\alpha}>0. \end{array}\right.
\end{equation*}
Writing $p_{\alpha,\epsilon}=p_{\alpha,\epsilon}(n):=n{\left(1-\epsilon-\frac{2}{\alpha}\right)}^{-1}$ (so that $s_{p_{\alpha,\epsilon}}+\epsilon+\frac{2}{\alpha}=0$), the conditions \eqref{ale-conditions} are equivalent to
\begin{equation*}
    \left\{\begin{array}{l} \text{either } \ell=p_{\alpha,\epsilon} \text{ and } 2\leq p_{\alpha,\epsilon}\leq\infty, \\ \text{or } \ell<p_{\alpha,\epsilon} \text{ and } 2<p_{\alpha,\epsilon}<\infty. \end{array}\right.
\end{equation*}
Defining $B_{\alpha,\ell,\epsilon}:=B$ if $\ell=p_{\alpha,\epsilon}$, or $B_{\alpha,\ell,\epsilon}:=\widetilde{B}$ if $\ell<p_{\alpha,\epsilon}$, it makes sense to consider, for $f\in\mathcal{S}'(\mathbb{R}^{n};\mathbb{R}^{n})$, the problem of finding a weakly* measurable function $u:(0,T)\rightarrow\mathcal{S}'(\mathbb{R}^{n};\mathbb{R}^{n})$ with ${\|u\|}_{X_{\ell,\epsilon}^{\alpha}(T)}<\infty$, which satisfies
\begin{equation}\label{besov-uniqueness-fixed-point}
    u(t) = S[f](t) - B_{\alpha,\ell,\epsilon}[u,u](t) \quad \text{in }\mathcal{S}'(\mathbb{R}^{n})\text{ for all }t\in(0,T).
\end{equation}
For assurance that our definitions are reasonable, the following proposition describes a situation in which the bilinear operators $B$ and $\widetilde{B}$ coincide, and describes a situation in which the operator $G\mathbb{P}\nabla\cdot$ coincides with a more classical operator involving the Oseen kernel.
\begin{proposition}\label{lorentz-besov-equiv}
    \begin{enumerate}[label=(\roman*)]
        \item Let $T\in(0,\infty)$, let $(\alpha,\ell,\epsilon)\in{[1,\infty]}^{2}\times(0,1]$ satisfy \eqref{ale-conditions}, let $p=p_{\alpha,\epsilon}$, and let $f\in\dot{B}_{p_{0},\infty}^{s_{0}}(\mathbb{R}^{n};\mathbb{R}^{n})$ for some $p_{0}\in[2,\infty]$ and $s_{0}\in(-1,0)$. Then any solution $u\in X_{\ell,\epsilon}^{\alpha}(T)$ to the equation \eqref{besov-uniqueness-fixed-point} belongs to $\widetilde{L}_{T}^{2}\dot{B}_{p_{0},1}^{0}(\mathbb{R}^{n})+\widetilde{L}_{T}^{\alpha}\dot{B}_{p,1}^{0}(\mathbb{R}^{n})$. Moreover, for any $u,v\in\widetilde{L}_{T}^{2}\dot{B}_{p_{0},1}^{0}(\mathbb{R}^{n})+\widetilde{L}_{T}^{\alpha}\dot{B}_{p,1}^{0}(\mathbb{R}^{n})$ we have $u\otimes v=\mathsf{Bony}(u,v)$ and $\mathbb{P}\nabla\cdot(u\otimes v)=\sum_{j\in\mathbb{Z}}\mathbb{P}\nabla\cdot\dot{\Delta}_{j}(u\otimes v)$.
        \item Let $T\in(0,\infty)$, $p\in[1,\infty]$, and $W\in L_{T}^{1}L^{p}(\mathbb{R}^{n};\mathbb{R}^{n\times n})$. Then $G\mathbb{P}\nabla\cdot W\in L_{T}^{1}L^{p}(\mathbb{R}^{n};\mathbb{R}^{n})$ can be represented by the (almost everywhere defined) measurable function on $(0,T)\times\mathbb{R}^{n}$ given by the expression
        \begin{equation*}
            \int_{0}^{t}\left(\int_{\mathbb{R}^{n}}\nabla_{k}\mathcal{T}_{ij}(t-s,x-y)[W_{jk}(s)](y)\,\mathrm{d}y\right)\mathrm{d}s,
        \end{equation*}
        where $\mathcal{T}_{ij}=\mathbb{P}_{ij}\Phi$ is the Oseen kernel (and $\Phi$ is the heat kernel).
    \end{enumerate}
\end{proposition}
For the sake of completeness, we will prove the following uniqueness result.
\begin{proposition}\label{besov-uniqueness-theorem}
    Let $T\in(0,\infty]$ and $f\in\mathcal{S}'(\mathbb{R}^{n};\mathbb{R}^{n})$. For any $(\alpha,\ell,\epsilon)\in{[1,\infty]}^{2}\times(0,1]$ satisfying \eqref{ale-conditions}, there exists at most one solution $u\in\cap_{T'\in(0,T)}X_{\ell,\epsilon}^{\alpha}(T')$ to the equation $u=S[f]-B_{\alpha,\ell,\epsilon}[u,u]$.
\end{proposition}
\begin{remark}\label{besov-uniqueness-remark}
\begin{enumerate}[label=(\roman*)]
    \item
    Applying Proposition \ref{besov-uniqueness-theorem} in the case $\alpha\in(2,\infty]$, $\ell=\infty$ and $\epsilon=1-\frac{2}{\alpha}$, we see that the fixed point problem $u=S[f]-B[u,u]$ satisfies uniqueness in $\cap_{T'\in(0,T)}\cup_{\alpha\in(2,\infty]}L_{T}^{\alpha}L^{\infty}(\mathbb{R}^{n})$.
    \item
    Proposition \ref{besov-uniqueness-theorem} in the case $\ell=p_{\alpha,\epsilon}$ was already proved in \cite{fabes1972}. A version in Lorentz spaces is proved in \cite{davies-lorentz-blowup}.
    \item
    For $n=3$, uniqueness in the class $\widetilde{L}_{T}^{\beta}\dot{B}_{\ell,q}^{s_{\ell}+\frac{2}{\beta}}(\mathbb{R}^{n})$, $\ell,q\in[1,\infty)$, $\beta\in(2,\infty)$, $s_{\ell\vee2}+\frac{2}{\beta}>0$, is proved in \cite[page 1419]{gallagher2003}. Writing $\frac{2}{\beta}=\epsilon+\frac{2}{\alpha}$ for $\epsilon\in(0,\frac{2}{\beta}]$, this implies uniqueness in the class $\widetilde{L}_{T}^{\alpha}\dot{B}_{\ell,q}^{s_{\ell}+\epsilon+\frac{2}{\alpha}}(\mathbb{R}^{n})$. Proposition \ref{besov-uniqueness-theorem} in the case $\ell<p_{\alpha,\epsilon}$ provides the new information that uniqueness holds in $\widetilde{L}_{T}^{\alpha}\dot{B}_{\ell,\infty}^{s_{\ell}+\epsilon+\frac{2}{\alpha}}(\mathbb{R}^{n})$.
    \item
    Under additional assumptions on the initial data, the paper \cite{chemin1999} provides uniqueness results in classes which are closely related to spaces of the form $\widetilde{L}_{T}^{\beta}\dot{B}_{\ell,\infty}^{s_{\ell}+\frac{2}{\beta}}(\mathbb{R}^{n})$.
\end{enumerate}
\end{remark}
The main result of this paper is the following theorem, which establishes properties of solutions when the initial data belongs to a subcritical Besov space $\dot{B}_{\infty,\infty}^{-1+\epsilon}(\mathbb{R}^{n};\mathbb{R}^{n})$, $\epsilon\in(0,1)$. As noted in Remark \ref{besov-existence-remark}(ii) below, the following theorem is an improvement on a result by Chemin and Gallagher \cite[Theorem 1.3]{chemin2019a}: for $n=3$, $\epsilon\in(0,1)$ and $f\in\dot{B}_{\infty,\infty}^{-1+\epsilon}(\mathbb{R}^{n};\mathbb{R}^{n})$, the result in \cite{chemin2019a} constructs a local solution in the class $Z_{\infty,\infty}^{-1+\epsilon}(T)$; we show that this solution and its maximal existence time are independent of $\epsilon$, and we establish additional regularity properties of the solution, depending on the Besov spaces to which the initial data belongs.
\begin{theorem}\label{besov-local-existence-theorem}
\begin{enumerate}[label=(\roman*)]
    \item
    There exists a positive constant $C_{\varphi}$ (where $\varphi$ is the cutoff function used to define the Littlewood-Paley projections), and there exists a positive constant $\widetilde{C}_{\varphi,\epsilon}$ for every $\epsilon\in(0,1)$, such that: if $T\in(0,\infty)$, $\epsilon\in(0,1)$, $s\in(-1,\infty)$, $p,q\in[1,\infty]$, and $f\in\dot{B}_{p,q}^{s}(\mathbb{R}^{n})\cap\dot{B}_{\infty,\infty}^{-1+\epsilon}(\mathbb{R}^{n})$ satisfies
    \begin{equation}\label{besov-local-existence-condition}
        \frac{C_{\varphi}}{\epsilon(1-\epsilon)}T^{\frac{\epsilon}{2}}{\|f\|}_{\dot{B}_{\infty,\infty}^{-1+\epsilon}(\mathbb{R}^{n})}\wedge\widetilde{C}_{\varphi,\epsilon}T^{\frac{1}{2}}{\|f\|}_{L^{\infty}(\mathbb{R}^{n})} < 1,
    \end{equation}
    then there exists $u\in Z_{p,q}^{s}(T)\cap Z_{\infty,\infty}^{-1+\epsilon}(T)$ satisfying $u(t)=S[f](t)-B[u,u](t)$ in $\mathcal{S}'(\mathbb{R}^{n})$ for all $t\in(0,T)$.
    \item
    Let $\epsilon\in(0,1)$, $s\in(-1,\infty)$, $p,q\in[1,\infty]$, and $f\in\dot{B}_{p,q}^{s}(\mathbb{R}^{n})\cap\dot{B}_{\infty,\infty}^{-1+\epsilon}(\mathbb{R}^{n})$. Let $T_{f}^{*}\in(0,\infty]$ be the maximal time for which there exists a solution $u\in\cap_{T'\in(0,T_{f}^{*})}\left(Z_{p,q}^{s}(T')\cap Z_{\infty,\infty}^{-1+\epsilon}(T')\right)$ to the equation $u=S[f]-B[u,u]$. If $T_{f}^{*}<\infty$, then we have the blowup estimate
    \begin{equation}\label{besov-blowup-estimate}
        \frac{C_{\varphi}}{\epsilon(1-\epsilon)}{(T_{f}^{*}-t)}^{\frac{\epsilon}{2}}{\|u(t)\|}_{\dot{B}_{\infty,\infty}^{-1+\epsilon}(\mathbb{R}^{n})}\wedge\widetilde{C}_{\varphi,\epsilon}{(T_{f}^{*}-t)}^{\frac{1}{2}}{\|u(t)\|}_{L^{\infty}(\mathbb{R}^{n})} \geq 1
    \end{equation}
    for all $t\in(0,T_{f}^{*})$, where $C_{\varphi}$ and $\widetilde{C}_{\varphi,\epsilon}$ are the constants from \eqref{besov-local-existence-condition}.
\end{enumerate}
\end{theorem}
\begin{remark}\label{besov-existence-remark}
\begin{enumerate}[label=(\roman*)]
    \item
    For $T\in(0,\infty)$ and $\epsilon\in(0,1)$, we have the embedding $Z_{\infty,\infty}^{-1+\epsilon}(T)\subseteq X_{\infty,\epsilon/2}^{4/(2-\epsilon)}(T)=L_{T}^{4/(2-\epsilon)}L^{\infty}(\mathbb{R}^{n})$ (proved in Lemma \ref{bilinear-estimates-lemma} below), so by Remark \ref{besov-uniqueness-remark}(i) the solution is unique among all solutions belonging to $\cap_{T'\in(0,T_{f}^{*})}\cup_{\alpha\in(2,\infty)}L_{T'}^{\alpha}L^{\infty}(\mathbb{R}^{n})$.
    \item
    For $n=3$, $\epsilon\in(0,1)$ and $f\in\dot{B}_{\infty,\infty}^{-1+\epsilon}(\mathbb{R}^{n})$, it was proved in \cite[Theorem 1.3]{chemin2019a} that there exists a local solution $u\in\cap_{T'\in(0,T_{f,\epsilon}^{*})}Z_{\infty,\infty}^{-1+\epsilon}(T')$, with maximal existence time $T_{f,\epsilon}^{*}\gtrsim_{\varphi,\epsilon}{\|f\|}_{\dot{B}_{\infty,\infty}^{-1+\epsilon}(\mathbb{R}^{n})}^{-2/\epsilon}$; as pointed out in \cite{chemin2019b}, this implies the blowup estimate ${\|u(t)\|}_{\dot{B}_{\infty,\infty}^{-1+\epsilon}(\mathbb{R}^{n})}\gtrsim_{\varphi,\epsilon}{(T_{f,\epsilon}^{*}-t)}^{-\epsilon/2}$ if $T_{f,\epsilon}^{*}$ is finite.

    Our Theorem \ref{besov-local-existence-theorem} goes even further. Besides obtaining a more explicit dependence on $\epsilon$ (see Remark \ref{bilinear-estimates-remark}(ii) which compares our slightly stronger bilinear estimate with the key bilinear estimate used in the proof of \cite[Theorem 1.3]{chemin2019a}), the blowup estimate ${\|u(t)\|}_{L^{\infty}(\mathbb{R}^{n})}\gtrsim_{\varphi,\epsilon}{(T_{f}^{*}-t)}^{-1/2}$ rules out the existence of a solution in $L_{T_{f}^{*}}^{2}L^{\infty}(\mathbb{R}^{n})$ if $T_{f}^{*}$ is finite, and demonstrates that the blowup time $T_{f}^{*}$ is independent of $\epsilon,s,p,q$. In particular, $T_{f}^{*}$ coincides with the maximal existence time for a solution in the class $\cap_{T'\in(0,T_{f}^{*})}\cup_{\alpha\in(2,\infty)}L_{T'}^{\alpha}L^{\infty}(\mathbb{R}^{n})$. This independence is sometimes referred to as {\em propagation of regularity}: for $f\in\dot{B}_{\infty,\infty}^{-1+\epsilon}(\mathbb{R}^{n})$, there exists a solution $u\in\cap_{T'\in(0,T_{f}^{*})}Z_{\infty,\infty}^{-1+\epsilon}(T')$ (which satisfies the uniqueness described in part (i) of this Remark); if we make the additional assumption that $f\in\dot{B}_{p,q}^{s}(\mathbb{R}^{n})$, then the solution satisfies the additional property $u\in\cap_{T'\in(0,T_{f}^{*})}Z_{p,q}^{s}(T')$.
    \item
    In the case $s=s_{p}+\epsilon$, the embedding $\dot{B}_{p,q}^{s_{p}+\epsilon}(\mathbb{R}^{n})\hookrightarrow\dot{B}_{\infty,\infty}^{-1+\epsilon}(\mathbb{R}^{n})$ (see \eqref{besov-embedding}, which similarly implies that $Z_{p,q}^{s_{p}+\epsilon}(T) \hookrightarrow Z_{\infty,\infty}^{-1+\epsilon}(T)$ for $T>0$) and the blowup estimate \eqref{besov-blowup-estimate} imply that ${\|u(t)\|}_{\dot{B}_{p,q}^{s_{p}+\epsilon}(\mathbb{R}^{n})}\gtrsim_{\varphi}\epsilon(1-\epsilon){(T_{f}^{*}-t)}^{-\epsilon/2}$ for all $t\in(0,T_{f}^{*})$ if $T_{f}^{*}<\infty$, so $T_{f}^{*}$ (if finite) is the first time at which the norm ${\|u(t)\|}_{\dot{B}_{p,q}^{s_{p}+\epsilon}(\mathbb{R}^{n})}$ blows up.
    \item
    In the critical case $\epsilon=0$, the blowup criteria become slightly more qualitative. In the case $n=3$ it is known (for example, from \cite[Theorem A.1]{gallagher2003}) that if $p,q\in[1,\infty)$ and $f\in\dot{B}_{p,q}^{s_{p}}(\mathbb{R}^{3})$, then there exists (at least locally in time) a unique solution $u\in C([0,T);\dot{B}_{p,q}^{s_{p}}(\mathbb{R}^{3}))\cap\bigcap_{T'\in(0,T),\,\beta\in[1,\infty)}\widetilde{L}_{T'}^{\beta}\dot{B}_{p,q}^{s_{p}+\frac{2}{\beta}}(\mathbb{R}^{3})$. Writing $T^{*}$ to denote the maximal existence time, it was established in \cite{gallagher2003} that if $T^{*}<\infty$ then we have the blowup criterion $u\notin C([0,T];\dot{B}_{p,q}^{s_{p}}(\mathbb{R}^{3}))$. This blowup criterion has since been improved to $\lim_{t\nearrow T^{*}}{\|u(t)\|}_{\dot{B}_{p,\infty}^{s_{p}}(\mathbb{R}^{3})}=\infty$ \cite[Corollary 1.9]{albritton2019}. (Intermediate improvements to the blowup criterion had been made in \cite{gallagher2016}, and then in \cite{albritton2018}.).
    \item
    Further properties of solutions may be established using the results of \cite[Section 8.4]{lemarie2016}. For example, \cite[Theorem 8.6]{lemarie2016} tells us, for $p\in(n,\infty]\cap[2,\infty]$, $\epsilon\in(0,-s_{p})$ and $q\in[\frac{-2}{s_{p}+\epsilon},\infty]$, that the bilinear estimate
    \begin{equation*}
        {\|B[u,v]\|}_{\mathcal{B}_{p,q}^{s_{p}+\epsilon}(T)}\lesssim_{\varphi,\epsilon,p,q}T^{\frac{\epsilon}{2}}{\|u\|}_{\mathcal{B}_{p,q}^{s_{p}+\epsilon}(T)}{\|v\|}_{\mathcal{B}_{p,q}^{s_{p}+\epsilon}(T)}
    \end{equation*}
    holds, where ${\|u\|}_{\mathcal{B}_{p,q}^{s}(T)}:={\left\|t^{-s/2}{\|u(t)\|}_{L^{p}(\mathbb{R}^{n})}\right\|}_{L^{q}((0,T),\frac{\mathrm{d}t}{t})}$. By making use of the heat estimate ${\|S[f]\|}_{\mathcal{B}_{p,q}^{s_{p}+\epsilon}(T)}\lesssim_{\varphi,\epsilon,p,q}{\|f\|}_{\dot{B}_{p,q}^{s_{p}+\epsilon}(\mathbb{R}^{n})}$, the existence time for a solution in $\mathcal{B}_{p,q}^{s_{p}+\epsilon}$ can then be bounded below by a constant multiple of ${\|f\|}_{\dot{B}_{p,q}^{s_{p}+\epsilon}(\mathbb{R}^{n})}^{-2/\epsilon}$.
\end{enumerate}
\end{remark}
To bridge some of the gaps between Proposition \ref{besov-uniqueness-theorem} and Theorem \ref{besov-local-existence-theorem}, the following corollary describes situations in which solutions in the uniqueness class $X_{\ell,\epsilon}^{\alpha}$ coincide with, and have the same existence time as, the solution from Theorem \ref{besov-local-existence-theorem}.
\begin{corollary}\label{further-uni}
    Let $(\alpha,\ell,\epsilon)\in{[1,\infty]}^{2}\times(0,1]$ satisfy \eqref{ale-conditions}, let $\eta\in(0,1)\cap(0,\epsilon+\frac{2}{\alpha}]$, and let $f\in\dot{B}_{\ell,\infty}^{s_{\ell}+\eta}(\mathbb{R}^{n})$. Let $\widetilde{T}\in[0,\infty]$ be the maximal time for which there exists a solution $v\in\cap_{T'\in(0,\widetilde{T})}X_{\ell,\epsilon}^{\alpha}(T')$ to the equation $v=S[f]-B_{\alpha,\ell,\epsilon}[v,v]$ (if no such solution exists, then set $\widetilde{T}=0$).
    \begin{enumerate}[label=(\roman*)]
        \item If $-\left(s_{\ell}+\frac{2}{\alpha}\right)<\epsilon\leq\eta$, or if $\epsilon<\eta<\epsilon+\frac{2}{\alpha}$, then $\widetilde{T}>0$.
        \item If $\eta<s_{\ell}+2\left(\epsilon+\frac{2}{\alpha}\right)$ and $\widetilde{T}>0$, then $\widetilde{T}=T_{f}^{*}$, and the solution $v\in\cap_{T'\in(0,\widetilde{T})}X_{\ell,\epsilon}^{\alpha}(T')$ coincides with the solution $u\in\cap_{T'\in(0,T_{f}^{*})}Z_{\ell,\infty}^{s_{\ell}+\eta}(T')$ from Theorem \ref{besov-local-existence-theorem}.
    \end{enumerate}
\end{corollary}
\begin{remark}\label{further-uni-remark}
\begin{enumerate}[label=(\roman*)]
    \item
    In the context of Corollary \ref{further-uni}, we have $f\in\dot{B}_{\ell,\infty}^{s_{\ell}+\eta}(\mathbb{R}^{n})\hookrightarrow\dot{B}_{\infty,\infty}^{-1+\eta}(\mathbb{R}^{n})$, so the conditions of Proposition \ref{lorentz-besov-equiv}(i) are satisfied. In particular, any solution $v\in\cap_{T'\in(0,\widetilde{T})}X_{\ell,\epsilon}^{\alpha}(T')$ to the equation $v=S[f]-B_{\alpha,\ell,\epsilon}[v,v]$ satisfies $B[v,v]=\widetilde{B}[v,v]$, while the solution \linebreak $u\in\cap_{T'\in(0,T_{f}^{*})}Z_{\ell,\infty}^{s_{\ell}+\eta}(T')\subseteq
    \cap_{T'\in(0,T_{f}^{*})}Z_{\infty,\infty}^{-1+\eta}(T')\subseteq
    \cap_{T'\in(0,T_{f}^{*})}X_{\infty,\eta/2}^{4/(2-\eta)}(T')$ (noting the embeddings from (i) and (iii) of Remark \ref{besov-existence-remark}) from Theorem \ref{besov-local-existence-theorem}  similarly satisfies $B[u,u]=\widetilde{B}[u,u]$.
    \item
    In the case $\epsilon=\eta$, Corollary \ref{further-uni} describes some of the connections between the uniqueness class $X_{\ell,\epsilon}^{\alpha}$ and the existence class $Z_{\ell,\infty}^{s_{\ell}+\epsilon}$, which are natural classes to consider when the initial data belongs to $\dot{B}_{\ell,\infty}^{s_{\ell}+\epsilon}(\mathbb{R}^{n})$.
\end{enumerate}
\end{remark}
The results of this paper are related to those in \cite{davies-lorentz-blowup}, in which we consider Lorentz spaces $L^{p,q}(\mathbb{R}^{n})$ instead of Besov spaces $\dot{B}_{p,q}^{s}(\mathbb{R}^{n})$. By the heat estimate ${\|e^{t\Delta}f\|}_{L^{\infty}(\mathbb{R}^{n})}\lesssim_{n,p}t^{-n/2p}{\|f\|}_{L^{p,q}(\mathbb{R}^{n})}$ for $p>1$ and $t>0$, and by the equivalence of norms ${\|f\|}_{\dot{B}_{\infty,\infty}^{s}(\mathbb{R}^{n})}\approx_{\varphi,s}\sup_{t\in(0,\infty)}t^{-s/2}{\|e^{t\Delta}f\|}_{L^{\infty}(\mathbb{R}^{n})}$ for $s<0$, we have the embedding $L^{p,q}(\mathbb{R}^{n})\hookrightarrow\dot{B}_{\infty,\infty}^{-n/p}(\mathbb{R}^{n})$ for $p\in(1,\infty)$. The solutions constructed in this paper coincide with the solutions constructed in \cite{davies-lorentz-blowup}, by virtue of uniqueness in $L_{T}^{\alpha}L^{\infty}(\mathbb{R}^{n})$ for $\alpha>2$. For $p\in(n,\infty)$, the blowup estimate \eqref{besov-blowup-estimate} implies $\inf_{t\in(0,T_{f}^{*})}{(T_{f}^{*}-t)}^{\frac{1}{2}\left(1-\frac{n}{p}\right)}{\|u(t)\|}_{L^{p,q}(\mathbb{R}^{n})}\gtrsim_{n,p}1$ if $T_{f}^{*}$ is finite. In \cite{davies-lorentz-blowup}, we establish the additional property that if the initial data belongs to a Lorentz space, then the solution remains bounded in that Lorentz space on all intervals $(0,T)$ with $T<T_{f}^{*}$. In particular, $T_{f}^{*}$ (if finite) is the first time at which the subcritical $(p>n)$ Lorentz norms blow up. In the Lebesgue case $L^{p}(\mathbb{R}^{n})=L^{p,p}(\mathbb{R}^{n})$, this reduces to the famous blowup result of Leray \cite{leray1934}. We note that the weaker blowup estimate $\limsup_{t\nearrow T_{f}^{*}}{(T_{f}^{*}-t)}^{\frac{1}{2}\left(1-\frac{n}{p}\right)}{\|u(t)\|}_{L^{p,q}(\mathbb{R}^{n})}\gtrsim_{n,p}1$ (if $T_{f}^{*}<\infty$) was proved in \cite[Theorem 1.4]{wang2021} for $p\geq n$.

For all $\epsilon\in(0,1)$ and $p,q\in[1,\infty]$, this paper establishes the blowup estimate ${\|u(t)\|}_{\dot{B}_{p,q}^{s_{p}+\epsilon}(\mathbb{R}^{n})}\gtrsim_{\varphi,\epsilon}{(T_{f}^{*}-t)}^{-\epsilon/2}$ for all $t\in(0,T_{f}^{*})$ if $T_{f}^{*}<\infty$. In \cite{davies-high-reg-besov}, we establish (using a method based on the type employed in, e.g., \cite{mccormick2016, robinson2014, robinson2012},  which  differs significantly from the method used in this present paper) the same blowup estimate (assuming nice regularity properties of the solution, and with the implied constant also depending on $p\vee q\vee 2$) in the case $n\geq3$, $\epsilon\in[1,2)$, $p,q\in[1,\frac{n}{2-\epsilon})$, and in the case $n\geq3$, $\epsilon=2$, $q=1$, $p\in[1,\infty)$.

The rest of this paper is organised as follows. In Section \ref{besov-preliminaries} we recall standard properties of Besov spaces and weakly* measurable functions. In Section \ref{estimates-for-the-fixed-point-problem} we prove the main estimates that we will need for the operators $S$, $B$ and $\widetilde{B}$. In Section \ref{besov-uniqueness-section} we prove Proposition \ref{besov-uniqueness-theorem}. In Section \ref{besov-existence-section} we prove Theorem \ref{besov-local-existence-theorem} and Corollary \ref{further-uni}. In the Appendix, we prove some technical lemmas which are used in Section 3, and we give a proof of Proposition \ref{lorentz-besov-equiv}.
\section{Preliminaries}\label{besov-preliminaries}
\subsection{Besov spaces}\label{besov-section}
We give an overview of Besov spaces, using \cite{bahouri2011} as our main reference.
\begin{lemma}\label{bahouri-prop-2.10} \cite[Proposition 2.10]{bahouri2011}.
    Let $\mathcal{C}=\mathcal{C}(n)$ be the annulus $B(0,8/3)\setminus\overline{B}(0,3/4)$ in $\mathbb{R}^{n}$. Then the set $\widetilde{\mathcal{C}}=B(0,2/3)+\mathcal{C}$ is an annulus, and there exist radial functions $\chi\in\mathcal{D}(B(0,4/3))$ and $\varphi\in\mathcal{D}(\mathcal{C})$, taking values in $[0,1]$, such that
    \begin{equation*}
        \left\{\begin{array}{ll}
            \chi(\xi)+\sum_{j\geq0}\varphi(2^{-j}\xi)=1 & \forall\,\xi\in\mathbb{R}^{n}, \\
            \sum_{j\in\mathbb{Z}}\varphi(2^{-j}\xi)=1 & \forall\,\xi\in\mathbb{R}^{n}\setminus\{0\}, \\
            |j-j'|\geq2\Rightarrow\supp\varphi(2^{-j}\cdot)\cap\supp\varphi(2^{-j'}\cdot)=\emptyset, \\
            j\geq1\Rightarrow\supp\chi\cap\supp\varphi(2^{-j}\cdot)=\emptyset, \\
            |j-j'|\geq5\Rightarrow2^{j'}\widetilde{\mathcal{C}}\cap2^{j}\mathcal{C}=\emptyset, \\
            1/2\leq\chi^{2}(\xi)+\sum_{j\geq0}\varphi^{2}(2^{-j}\xi)\leq1 & \forall\,\xi\in\mathbb{R}^{n}, \\
            1/2\leq\sum_{j\in\mathbb{Z}}\varphi^{2}(2^{-j}\xi)\leq1 & \forall\,\xi\in\mathbb{R}^{n}\setminus\{0\}. \\
        \end{array}\right.
    \end{equation*}
\end{lemma}
We fix $\chi,\varphi$ satisfying Lemma \ref{bahouri-prop-2.10}. For $j\in\mathbb{Z}$ and $u\in\mathcal{S}'(\mathbb{R}^{n})$, we define
\begin{equation*}
    \dot{S}_{j}u := \chi(2^{-j}D)u = \mathcal{F}^{-1}\chi(2^{-j}\xi)\mathcal{F}u,
\end{equation*}
\begin{equation*}
    \dot{\Delta}_{j}u := \varphi(2^{-j}D)u = \mathcal{F}^{-1}\varphi(2^{-j}\xi)\mathcal{F}u,
\end{equation*}
where we recall that the Fourier transform of a compactly supported distribution is a smooth function. We recall the following useful inequalities.
\begin{lemma}\label{useful-inequalities} \cite[Lemmas 2.1-2.4, Remark 2.11]{bahouri2011}.
    Let $\rho$ be a smooth function on $\mathbb{R}^{n}\setminus\{0\}$ which is positive homogeneous of degree $\alpha\in\mathbb{R}$ (meaning that $\rho(\lambda x)=\lambda^{\alpha}\rho(x)$ for all $\lambda\in(0,\infty)$ and $x\in\mathbb{R}^{n}\setminus\{0\}$). Then for all $j\in\mathbb{Z}$, $u\in\mathcal{S}'(\mathbb{R}^{n})$, $t\in(0,\infty)$ and $1\leq p\leq q\leq\infty$ we have
    \begin{equation}\label{useful-inequality-1}
        {\|\dot{S}_{j}u\|}_{L^{p}(\mathbb{R}^{n})}\vee{\|\dot{\Delta}_{j}u\|}_{L^{p}(\mathbb{R}^{n})} \lesssim_{\varphi} {\|u\|}_{L^{p}(\mathbb{R}^{n})},
    \end{equation}
    \begin{equation}\label{useful-inequality-2}
        {\|\rho(D)\dot{\Delta}_{j}u\|}_{L^{q}(\mathbb{R}^{n})} \lesssim_{\rho} 2^{j\alpha}2^{j\left(\frac{n}{p}-\frac{n}{q}\right)}{\|\dot{\Delta}_{j}u\|}_{L^{p}(\mathbb{R}^{n})},
    \end{equation}
    \begin{equation}\label{useful-inequality-3}
        {\|\dot{\Delta}_{j}u\|}_{L^{p}(\mathbb{R}^{n})} \lesssim_{n} 2^{-j}{\|\nabla\dot{\Delta}_{j}u\|}_{L^{p}(\mathbb{R}^{n})},
    \end{equation}
    \begin{equation}\label{basic-heat-estimate}
        {\|e^{t\Delta}\dot{\Delta}_{j}u\|}_{L^{p}(\mathbb{R}^{n})} \lesssim_{n} e^{-c_{n}t2^{2j}}{\|\dot{\Delta}_{j}u\|}_{L^{p}(\mathbb{R}^{n})}.
    \end{equation}
\end{lemma}
One can give meaning to the decomposition $u=\sum_{j\in\mathbb{Z}}\dot{\Delta}_{j}u$ in view of the following lemma.
\begin{lemma}\label{littlewood-paley-decomposition} \cite[Definition 1.26, Propositions 2.12-2.14]{bahouri2011}.
    If $u\in\mathcal{S}'(\mathbb{R}^{n})$, then $\dot{S}_{j}u\overset{j\rightarrow\infty}{\rightarrow}u$ in $\mathcal{S}'(\mathbb{R}^{n})$. Define
    \begin{equation*}
        \mathcal{S}_{h}'(\mathbb{R}^{n}) := \left\{u\in\mathcal{S}'(\mathbb{R}^{n})\text{ }:\text{ }{\|\dot{S}_{j}u\|}_{L^{\infty}(\mathbb{R}^{n})}\overset{j\rightarrow-\infty}{\rightarrow}0\right\}.
    \end{equation*}
    (For example, if $\mathcal{F}u$ is locally integrable near $\xi=0$, then $u\in\mathcal{S}_{h}'(\mathbb{R}^{n})$. The condition $u\in\mathcal{S}_{h}'(\mathbb{R}^{n})$ is independent of our choice of $\varphi$.). If $u\in\mathcal{S}_{h}'(\mathbb{R}^{n})$, then $u=\sum_{j\in\mathbb{Z}}\dot{\Delta}_{j}u$ in $\mathcal{S}'(\mathbb{R}^{n})$.
\end{lemma}
For $s\in\mathbb{R}$ and $p,q\in[1,\infty]$, we define the Besov seminorm
\begin{equation*}
    {\|u\|}_{\dot{B}_{p,q}^{s}(\mathbb{R}^{n})} := {\left\|j\mapsto2^{js}{\|\dot{\Delta}_{j}u\|}_{L^{p}(\mathbb{R}^{n})}\right\|}_{l^{q}(\mathbb{Z})} \quad \text{for }u\in\mathcal{S}'(\mathbb{R}^{n}),
\end{equation*}
where choosing a different function $\varphi$ yields an equivalent seminorm \cite[Remark 2.17]{bahouri2011}. We also define the Besov space
\begin{equation*}
    \dot{B}_{p,q}^{s}(\mathbb{R}^{n}) := \left\{u\in\mathcal{S}_{h}'(\mathbb{R}^{n})\text{ : }{\|u\|}_{\dot{B}_{p,q}^{s}(\mathbb{R}^{n})}<\infty\right\},
\end{equation*}
so that $\left(\dot{B}_{p,q}^{s}(\mathbb{R}^{n}),{\|\cdot\|}_{\dot{B}_{p,q}^{s}(\mathbb{R}^{n})}\right)$ is a normed space \cite[Proposition 2.16]{bahouri2011}. Lemma \ref{useful-inequalities} and Lemma \ref{littlewood-paley-decomposition} yield the inequalities
\begin{equation}\label{besov-embedding}
    {\|u\|}_{\dot{B}_{p_{2},q_{2}}^{\frac{n}{p_{2}}+\epsilon}(\mathbb{R}^{n})} \lesssim_{n} {\|u\|}_{\dot{B}_{p_{1},q_{1}}^{\frac{n}{p_{1}}+\epsilon}(\mathbb{R}^{n})} \quad \text{for }p_{1}\leq p_{2},\,q_{1}\leq\,q_{2},\,\epsilon\in\mathbb{R},\,u\in\mathcal{S}'(\mathbb{R}^{n}),
\end{equation}
\begin{equation}\label{rough-lp}
    {\|u\|}_{\dot{B}_{p,\infty}^{0}(\mathbb{R}^{n})} \lesssim_{\varphi} {\|u\|}_{L^{p}(\mathbb{R}^{n})} \quad \text{for }u\in\mathcal{S}'(\mathbb{R}^{n}),
\end{equation}
\begin{equation}\label{smooth-lp}
    {\|u\|}_{L^{p}(\mathbb{R}^{n})} \leq {\|u\|}_{\dot{B}_{p,1}^{0}(\mathbb{R}^{n})} \quad \text{for }u\in\mathcal{S}_{h}'(\mathbb{R}^{n}).
\end{equation}
We also have, for $\lambda\in(0,1)$ and $u\in\mathcal{S}'(\mathbb{R}^{n})$, the interpolation inequalities
\begin{equation}\label{interpolation-holder}
    {\|u\|}_{\dot{B}_{\frac{p_{1}p_{2}}{\lambda p_{2}+(1-\lambda)p_{1}},\frac{q_{1}q_{2}}{\lambda q_{2}+(1-\lambda)q_{1}}}^{\lambda s_{1}+(1-\lambda)s_{2}}(\mathbb{R}^{n})} \leq {\|u\|}_{\dot{B}_{p_{1},q_{1}}^{s_{1}}(\mathbb{R}^{n})}^{\lambda}{\|u\|}_{\dot{B}_{p_{2},q_{2}}^{s_{2}}(\mathbb{R}^{n})}^{1-\lambda},
\end{equation}
\begin{equation}\label{interpolation-geometric}
    {\|u\|}_{\dot{B}_{p,1}^{\lambda s_{1}+(1-\lambda)s_{2}}(\mathbb{R}^{n})} \lesssim \frac{1}{\lambda(1-\lambda)(s_{2}-s_{1})}{\|u\|}_{\dot{B}_{p,\infty}^{s_{1}}(\mathbb{R}^{n})}^{\lambda}{\|u\|}_{\dot{B}_{p,\infty}^{s_{2}}(\mathbb{R}^{n})}^{1-\lambda} \quad \text{for }s_{1}<s_{2},
\end{equation}
where \eqref{interpolation-holder} comes from H\"{o}lder's inequality, while \eqref{interpolation-geometric} comes from writing $\sum_{j\in\mathbb{Z}}=\sum_{j\leq j_{0}}+\sum_{j>j_{0}}$ with $2^{j_{0}(s_{2}-s_{1})}{\|u\|}_{\dot{B}_{p,\infty}^{s_{1}}(\mathbb{R}^{n})}\approx{\|u\|}_{\dot{B}_{p,\infty}^{s_{2}}(\mathbb{R}^{n})}$ and applying geometric series.

We now recall the following convergence properties.
\begin{lemma}\label{convergence-lemma} \cite[Lemma 2.23]{bahouri2011}.
    Let $\mathcal{C}'$ be an annulus and ${(u_{j})}_{j\in\mathbb{Z}}$ be a sequence of functions such that $\supp\mathcal{F}u_{j}\subseteq2^{j}\mathcal{C}'$ and ${\left\|j\mapsto2^{js}{\|u_{j}\|}_{L^{p}(\mathbb{R}^{n})}\right\|}_{l^{q}(\mathbb{Z})}<\infty$. If the series $\sum_{j\in\mathbb{Z}}u_{j}$ converges in $\mathcal{S}'(\mathbb{R}^{n})$ to some $u\in\mathcal{S}'(\mathbb{R}^{n})$, then
    \begin{equation*}\label{convergence-inequality}
        {\|u\|}_{\dot{B}_{p,q}^{s}(\mathbb{R}^{n})} \lesssim_{\varphi} C_{\mathcal{C}'}^{1+|s|}{\left\|j\mapsto2^{js}{\|u_{j}\|}_{L^{p}(\mathbb{R}^{n})}\right\|}_{l^{q}(\mathbb{Z})}.
    \end{equation*}
    Note: If $(s,p,q)$ satisfy the condition
    \begin{equation}\label{negative-scaling}
        s<\frac{n}{p}, \quad \text{or} \quad s=\frac{n}{p}\text{ and }q=1,
    \end{equation}
    then the hypothesis of convergence is satisfied, and $u\in\mathcal{S}_{h}'(\mathbb{R}^{n})$.
\end{lemma}
\begin{lemma}\label{Banach-space} \cite[Theorem 2.25]{bahouri2011}.
    If $(s_{1},p_{1},q_{1})$ satisfy \eqref{negative-scaling}, then the space $\dot{B}_{p_{1},q_{1}}^{s_{1}}(\mathbb{R}^{n})\cap\dot{B}_{p_{2},q_{2}}^{s_{2}}(\mathbb{R}^{n})$, when equipped with the norm ${\|u\|}_{\dot{B}_{p_{1},q_{1}}^{s_{1}}(\mathbb{R}^{n})\cap\dot{B}_{p_{2},q_{2}}^{s_{2}}(\mathbb{R}^{n})}:={\|u\|}_{\dot{B}_{p_{1},q_{1}}^{s_{1}}(\mathbb{R}^{n})}\vee{\|u\|}_{\dot{B}_{p_{2},q_{2}}^{s_{2}}(\mathbb{R}^{n})}$, is a Banach space.
\end{lemma}
A useful consequence of Lemma \ref{convergence-lemma} is that if $u\in\mathcal{S}'(\mathbb{R}^{n})$ satisfies ${\|u\|}_{\dot{B}_{p,q}^{s}(\mathbb{R}^{n})}<\infty$ for some $(s,p,q)$ satisfying \eqref{negative-scaling}, then $u\in\mathcal{S}_{h}'(\mathbb{R}^{n})$. For example, \cite[Theorem 2.34]{bahouri2011} states that the characterisation
\begin{equation}\label{besov-heat-characterisation}
    {\|u\|}_{\dot{B}_{\infty,\infty}^{s}(\mathbb{R}^{n})} \approx_{\varphi,s} \sup_{t\in(0,\infty)}t^{-s/2}{\|e^{t\Delta}u\|}_{L^{\infty}(\mathbb{R}^{n})} \quad \text{for }s<0
\end{equation}
holds for all $u\in\mathcal{S}_{h}'(\mathbb{R}^{n})$; Lemma \ref{convergence-lemma} allows us to extend \eqref{besov-heat-characterisation} to all $u\in\mathcal{S}'(\mathbb{R}^{n})$.

In duality with Lemma \ref{convergence-lemma}, the following lemma tells us that the Schwartz space $\mathcal{S}(\mathbb{R}^{n})$ embeds continuously into certain Besov spaces.
\begin{lemma}\label{convergence-duality}
    Define ${\|\phi\|}_{\alpha,\beta}:=\sup_{x\in\mathbb{R}^{n}}|x^{\alpha}\nabla^{\beta}\phi(x)|$ for $\phi\in\mathcal{S}(\mathbb{R}^{n})$ and $\alpha,\beta\in\mathbb{Z}_{\geq0}^{n}$. Then for any $(s,p,q)$ satisfying \eqref{negative-scaling}, there exists a finite subset $A_{s,p}\subseteq\mathbb{Z}_{\geq0}^{n}\times\mathbb{Z}_{\geq0}^{n}$ satisfying
    \begin{equation*}
        {\|\phi\|}_{\dot{B}_{p',q'}^{-s}(\mathbb{R}^{n})} \lesssim_{\varphi,s,p} \sum_{(\alpha,\beta)\in A_{s,p}}{\|\phi\|}_{\alpha,\beta} \quad \text{for }\phi\in\mathcal{S}(\mathbb{R}^{n}).
    \end{equation*}
\end{lemma}
\begin{proof}
    By \eqref{useful-inequality-3} and \eqref{rough-lp}, for all $k\in\mathbb{Z}_{\geq0}$ and $\phi\in\mathcal{S}(\mathbb{R}^{n})$ we have
    \begin{equation}\label{convergence-duality-proof}
        {\|\phi\|}_{\dot{B}_{1,\infty}^{k}(\mathbb{R}^{n})} \lesssim_{\varphi} \sum_{|\beta|=k}{\|\nabla^{\beta}\phi\|}_{L^{1}(\mathbb{R}^{n})} \lesssim_{n} \sum_{\substack{|\alpha|\leq n+1 \\ |\beta|=k}}{\|\phi\|}_{\alpha,\beta}.
    \end{equation}
    By the embedding \eqref{besov-embedding}, to prove the lemma it suffices to control ${\|\phi\|}_{\dot{B}_{1,q'}^{\frac{n}{p}-s}(\mathbb{R}^{n})}$. In the case $s=\frac{n}{p}$ and $q=1$, this means we need to control ${\|\phi\|}_{\dot{B}_{1,\infty}^{0}(\mathbb{R}^{n})}$, which is achieved by taking $k=0$ in \eqref{convergence-duality-proof}. In the case $s<\frac{n}{p}$, we choose an integer $k$ satisfying $k>\frac{n}{p}-s$, and we use \eqref{besov-embedding} and \eqref{interpolation-geometric} to estimate
    \begin{equation*}
        {\|\phi\|}_{\dot{B}_{1,q'}^{\frac{n}{p}-s}(\mathbb{R}^{n})} \leq {\|\phi\|}_{\dot{B}_{1,1}^{\frac{n}{p}-s}(\mathbb{R}^{n})} \lesssim_{n,s,p} {\|\phi\|}_{\dot{B}_{1,\infty}^{0}(\mathbb{R}^{n})\cap\dot{B}_{1,\infty}^{k}(\mathbb{R}^{n})},
    \end{equation*}
    where ${\|\phi\|}_{\dot{B}_{1,\infty}^{0}(\mathbb{R}^{n})\cap\dot{B}_{1,\infty}^{k}(\mathbb{R}^{n})}$ can be controlled using \eqref{convergence-duality-proof}.
\end{proof}
If $u\in\dot{B}_{p_{1},1}^{0}(\mathbb{R}^{n})$ and $v\in\dot{B}_{p_{2},1}^{0}(\mathbb{R}^{n})$ with $\frac{1}{p_{1}}+\frac{1}{p_{2}}\leq1$, then the series $uv=\sum_{(j,j')\in\mathbb{Z}^{2}}\dot{\Delta}_{j}u\,\dot{\Delta}_{j'}v$ converges absolutely in $L^{\frac{p_{1}p_{2}}{p_{1}+p_{2}}}(\mathbb{R}^{n})$, which justifies the Bony decomposition
\begin{equation*}
    uv = \dot{T}_{u}v+\dot{T}_{v}u+\dot{R}(u,v),
\end{equation*}
\begin{equation*}
    \dot{T}_{u}v = \sum_{j\in\mathbb{Z}}\dot{S}_{j-1}u\,\dot{\Delta}_{j}v,
\end{equation*}
\begin{equation*}
    \dot{R}(u,v) = \sum_{j\in\mathbb{Z}}\sum_{|\nu|\leq1}\dot{\Delta}_{j}u\,\dot{\Delta}_{j-\nu}v.
\end{equation*}
We will require the following estimates for the operators $\dot{T}$ and $\dot{R}$.
\begin{lemma}\label{paraproduct} \cite[Theorem 2.47]{bahouri2011}.
    Suppose that $s=s_{1}+s_{2}$, $p=\frac{p_{1}p_{2}}{p_{1}+p_{2}}$ and $q=\frac{q_{1}q_{2}}{q_{1}+q_{2}}$. Let $u,v\in\mathcal{S}'(\mathbb{R}^{n})$, and assume that the series $\sum_{j\in\mathbb{Z}}\dot{S}_{j-1}u\,\dot{\Delta}_{j}v$ converges in $\mathcal{S}'(\mathbb{R}^{n})$ to some $\dot{T}_{u}v\in\mathcal{S}'(\mathbb{R}^{n})$. Then
    \begin{equation}\label{bony-estimate-1}
        {\|\dot{T}_{u}v\|}_{\dot{B}_{p,q}^{s}(\mathbb{R}^{n})} \lesssim_{\varphi} C_{n}^{1+|s|}{\|u\|}_{L^{p_{1}}(\mathbb{R}^{n})}{\|v\|}_{\dot{B}_{p_{2},q}^{s}(\mathbb{R}^{n})},
    \end{equation}
    \begin{equation}\label{bony-estimate-2}
        {\|\dot{T}_{u}v\|}_{\dot{B}_{p,q}^{s}(\mathbb{R}^{n})} \lesssim_{\varphi} \frac{C_{n}^{1+|s|}}{-s_{1}}{\|u\|}_{\dot{B}_{p_{1},q_{1}}^{s_{1}}(\mathbb{R}^{n})}{\|v\|}_{\dot{B}_{p_{2},q_{2}}^{s_{2}}(\mathbb{R}^{n})} \quad \text{if } s_{1}<0.
    \end{equation}
    Note: If $(s,p,q)$ satisfy \eqref{negative-scaling}, and the right hand side of either \eqref{bony-estimate-1} or \eqref{bony-estimate-2} is finite, then the hypothesis of convergence is satisfied, and $\dot{T}_{u}v\in\mathcal{S}_{h}'(\mathbb{R}^{n})$.
\end{lemma}
\begin{lemma}\label{remainder} \cite[Theorem 2.52]{bahouri2011}.
    Suppose that $s=s_{1}+s_{2}$, $p=\frac{p_{1}p_{2}}{p_{1}+p_{2}}$ and $q=\frac{q_{1}q_{2}}{q_{1}+q_{2}}$. Let $u,v\in\mathcal{S}'(\mathbb{R}^{n})$, and assume that the series $\sum_{j\in\mathbb{Z}}\sum_{|\nu|\leq1}\dot{\Delta}_{j}u\,\dot{\Delta}_{j-\nu}v$ converges in $\mathcal{S}'(\mathbb{R}^{n})$ to some $\dot{R}(u,v)\in\mathcal{S}'(\mathbb{R}^{n})$. Then
    \begin{equation}\label{bony-estimate-3}
        {\|\dot{R}(u,v)\|}_{\dot{B}_{p,q}^{s}(\mathbb{R}^{n})} \lesssim_{\varphi} \frac{C_{n}^{1+|s|}}{s}{\|u\|}_{\dot{B}_{p_{1},q_{1}}^{s_{1}}(\mathbb{R}^{n})}{\|v\|}_{\dot{B}_{p_{2},q_{2}}^{s_{2}}(\mathbb{R}^{n})} \quad \text{if } s>0,
    \end{equation}
    \begin{equation}\label{bony-estimate-4}
        {\|\dot{R}(u,v)\|}_{\dot{B}_{p,\infty}^{s}(\mathbb{R}^{n})} \lesssim_{\varphi} C_{n}^{1+|s|}{\|u\|}_{\dot{B}_{p_{1},q_{1}}^{s_{1}}(\mathbb{R}^{n})}{\|v\|}_{\dot{B}_{p_{2},q_{2}}^{s_{2}}(\mathbb{R}^{n})} \quad \text{if }q=1\text{ and } s\geq0.
    \end{equation}
    Note: If $(s,p,q)$ satisfy \eqref{negative-scaling} and the right hand side of \eqref{bony-estimate-3} is finite, or if $(s,p,\infty)$ satisfy \eqref{negative-scaling} and the right hand side of \eqref{bony-estimate-4} is finite, then the hypothesis of convergence is satisfied, and $\dot{R}(u,v)\in\mathcal{S}_{h}'(\mathbb{R}^{n})$.
\end{lemma}
\subsection{Weak* measurability}\label{weak*-measurability}
For a measurable space $(E,\mathcal{E})$, a function $u:E\rightarrow\mathcal{D}'(\mathbb{R}^{n})$ is said to be {\em weakly* measurable} if the map $t\mapsto\langle u(t),\phi\rangle$ is measurable for all $\phi\in\mathcal{D}(\mathbb{R}^{n})$. By density of $\mathcal{D}(\mathbb{R}^{n})$ in $\mathcal{S}(\mathbb{R}^{n})$, a function $u:E\rightarrow\mathcal{S}'(\mathbb{R}^{n})$ is weakly* measurable if and only if the map $t\mapsto\langle u(t),\phi\rangle$ is measurable for all $\phi\in\mathcal{S}(\mathbb{R}^{n})$. If $u$ is a weakly* measurable function $u:E\rightarrow\mathcal{S}'(\mathbb{R}^{n})$, then $\dot{\Delta}_{j}u$ is a weakly* measurable function $\dot{\Delta}_{j}u:E\rightarrow L_{\mathrm{loc}}^{1}(\mathbb{R}^{n})$.

For any measure space $(E,\mathcal{E},\mu)$, and any $s\in\mathbb{R}$ and $\alpha,p,q\in[1,\infty]$, if $u:E\rightarrow\mathcal{S}'(\mathbb{R}^{n})$ is a weakly* measurable function satisfying $\dot{\Delta}_{j}u(t)\in L^{p}(\mathbb{R}^{n})$ for $\mu$-almost every $t\in E$ for all $j\in\mathbb{Z}$, then we define the seminorm
\begin{equation*}
    {\|u\|}_{\widetilde{L}^{\alpha}(E,\mathcal{E},\mu;\dot{B}_{p,q}^{s}(\mathbb{R}^{n}))} := {\left\|j\mapsto2^{js}{\|\dot{\Delta}_{j}u\|}_{L^{\alpha}(E,\mathcal{E},\mu;L^{p}(\mathbb{R}^{n}))}\right\|}_{l^{q}(\mathbb{Z})},
\end{equation*}
where measurability of the map $t\mapsto{\|\dot{\Delta}_{j}u(t)\|}_{L^{p}(\mathbb{R}^{n})}$ is verified in Lemma \ref{measurable-norms} in the appendix. We define the Chemin-Lerner space $\widetilde{L}^{\alpha}(E,\mathcal{E},\mu;\dot{B}_{p,q}^{s}(\mathbb{R}^{n}))$ to be the set of equivalence classes of weakly* measurable functions $u:E\rightarrow\mathcal{S}_{h}'(\mathbb{R}^{n})$ satisfying ${\|u\|}_{\widetilde{L}^{\alpha}(E,\mathcal{E},\mu;\dot{B}_{p,q}^{s}(\mathbb{R}^{n}))}<\infty$. Many of the results concerning the spaces $\dot{B}_{p,q}^{s}(\mathbb{R}^{n})$ carry over to the spaces $\widetilde{L}^{\alpha}(E,\mathcal{E},\mu;\dot{B}_{p,q}^{s}(\mathbb{R}^{n}))$, replacing ${\|\dot{\Delta}_{j}u\|}_{L^{p}(\mathbb{R}^{n})}$ with ${\|\dot{\Delta}_{j}u\|}_{L^{\alpha}(E,\mathcal{E},\mu;L^{p}(\mathbb{R}^{n}))}$ in the proofs, and replacing convergence $f_{m}\rightarrow f$ in $\mathcal{S}'(\mathbb{R}^{n})$ with convergence $f_{m}(t)\rightarrow f(t)$ in $\mathcal{S}'(\mathbb{R}^{n})$ for (at least) $\mu$-almost every $t\in E$. We write $\widetilde{L}^{\overline{\infty}}$ when we want to replace ${\|\cdot\|}_{L^{\infty}(E,\mathcal{E},\mu)}$ with the supremum norm. We write $\widetilde{L}_{T}^{\alpha}\dot{B}_{p,q}^{s}(\mathbb{R}^{n})=\widetilde{L}^{\alpha}((0,T);\dot{B}_{p,q}^{s}(\mathbb{R}^{n}))$ when the interval $(0,T)$ is equipped with the Borel sigma-algebra and Lebesgue measure. By Minkowski's inequality we have the following embeddings, which will be used frequently in this paper:
\begin{equation}\label{chemin-lerner-minkowski}
    \widetilde{L}_{T}^{\alpha}\dot{B}_{p,q}^{s}(\mathbb{R}^{n})\hookrightarrow L_{T}^{\alpha}\dot{B}_{p,q}^{s}(\mathbb{R}^{n}) \text{ if }q\leq\alpha, \qquad L_{T}^{\alpha}\dot{B}_{p,q}^{s}(\mathbb{R}^{n})\hookrightarrow \widetilde{L}_{T}^{\alpha}\dot{B}_{p,q}^{s}(\mathbb{R}^{n}) \text{ if }q\geq\alpha.
\end{equation}
When $\alpha = \infty$, \eqref{chemin-lerner-minkowski} is similarly true with  $L_{T}^{\infty}$/$\widetilde{L}_{T}^{\infty}$ replaced by  $L_{T}^{\overline{\infty}}$/$\widetilde{L}_{T}^{\overline{\infty}}$.
\section{Estimates for the fixed point problem}\label{estimates-for-the-fixed-point-problem}
\subsection{The heat operator}
If $f\in\mathcal{S}'(\mathbb{R}^{n})$, then $S[f](t)=e^{t\Delta}f=\mathcal{F}^{-1}e^{-t{|\xi|}^{2}}\mathcal{F}f$ defines a continuous (and hence weakly* measurable) function $S[f]:(0,\infty)\rightarrow\mathcal{S}'(\mathbb{R}^{n})$. We have the semigroup property $e^{s\Delta}e^{t\Delta}=e^{(s+t)\Delta}$, and the identity
\begin{equation*}
    \langle e^{t\Delta}f,\phi\rangle = \langle f,\Phi(t)*\phi\rangle \quad \text{for all }f\in\mathcal{S}'(\mathbb{R}^{n})\text{ and }\phi\in\mathcal{S}(\mathbb{R}^{n}),
\end{equation*}
where $\Phi$ is the heat kernel. The heat map $e^{t\Delta}$ commutes with the operators $\dot{S}_{j}$ and $\dot{\Delta}_{j}$, so if $f\in\mathcal{S}_{h}'(\mathbb{R}^{n})$ then $e^{t\Delta}f\in\mathcal{S}_{h}'(\mathbb{R}^{n})$. By estimate \eqref{basic-heat-estimate} we have, for $\kappa_{j}(t)=\kappa_{j,n}(t):=e^{-c_{n}t2^{2j}}$,
\begin{equation}\label{kappa}
    {\|e^{t\Delta}f\|}_{L^{p}(\mathbb{R}^{n})} \lesssim_{n} \kappa_{j}(t){\|f\|}_{L^{p}(\mathbb{R}^{n})}, \quad \kappa_{j}(t)\leq1, \quad {\|\kappa_{j}\|}_{L_{T}^{\alpha}}\lesssim_{n}2^{-\frac{2j}{\alpha}}\wedge T^{\frac{1}{\alpha}},
\end{equation}
so we have
\begin{equation}\label{besov-heat-estimate-0}
    {\|\dot{\Delta}_{j}S[f]\|}_{L_{T}^{\alpha}L^{p}(\mathbb{R}^{n})} \lesssim_{n} (2^{-\frac{2j}{\alpha}}\wedge T^\frac{1}{\alpha}){\|\dot{\Delta}_{j}f\|}_{L^{p}(\mathbb{R}^{n})}, \quad {\|\dot{\Delta}_{j}S[f]\|}_{L_{T}^{\overline{\infty}}L^{p}(\mathbb{R}^{n})} \lesssim_{n} {\|\dot{\Delta}_{j}f\|}_{L^{p}(\mathbb{R}^{n})},
\end{equation}
from which we deduce
\begin{equation}\label{besov-heat-estimate-1}
    {\|S[f]\|}_{\widetilde{L}_{T}^{\alpha}\dot{B}_{p,q}^{s+\frac{2}{\alpha}}(\mathbb{R}^{n})} + {\|S[f]\|}_{\widetilde{L}_{T}^{\overline{\infty}}\dot{B}_{p,q}^{s}(\mathbb{R}^{n})} \lesssim_{n} {\|f\|}_{\dot{B}_{p,q}^{s}(\mathbb{R}^{n})} \quad \text{for }\alpha\in[1,\infty).
\end{equation}
On the other hand, Young's convolution inequality yields
\begin{equation}\label{besov-heat-estimate-3}
    {\|S[f](t)\|}_{L^{p}(\mathbb{R}^{n})} \leq {\|\Phi(t)\|}_{L^{1}(\mathbb{R}^{n})}{\|f\|}_{L^{p}(\mathbb{R}^{n})} = {\|f\|}_{L^{p}(\mathbb{R}^{n})}.
\end{equation}
\subsection{The bilinear operator}
For the bilinear operators $B$ and $\widetilde{B}$, we use the notation
\begin{equation*}
    {(u\otimes v)}_{ij} := u_{i}v_{j} \quad \text{(Pointwise product),}
\end{equation*}
\begin{equation*}
    {(\mathsf{Bony}(u,v))}_{ij} := \dot{T}_{u_{i}}v_{j} + \dot{T}_{v_{j}}u_{i} + \dot{R}(u_{i},v_{j}) \quad \text{(Bony product),}
\end{equation*}
\begin{equation*}
    \mathbb{P}_{ij}\nabla_{k}\phi := \mathcal{F}^{-1}\left[\xi\mapsto\mathrm{i}\frac{\xi_{i}\xi_{j}\xi_{k}}{{|\xi|}^{2}}\mathcal{F}\phi(\xi)\right] \quad \text{for }\phi\in\mathcal{S}(\mathbb{R}^{n}),
\end{equation*}
\begin{equation*}
    \langle{(\mathbb{P}\nabla\cdot W)}_{i},\phi\rangle := -\langle W_{jk},\mathbb{P}_{ij}\nabla_{k}\phi\rangle \quad \text{for }\phi\in\mathcal{S}(\mathbb{R}^{n}),
\end{equation*}
\begin{equation*}
    {(\mathbb{P}\nabla\cdot)}_{\varphi}W := \sum_{j\in\mathbb{Z}}\mathbb{P}\nabla\cdot\dot{\Delta}_{j}W,
\end{equation*}
\begin{equation*}
    \langle G[w](t),\phi\rangle := \int_{0}^{t}\langle e^{(t-s)\Delta}w(s),\phi\rangle\,\mathrm{d}s \quad \text{for }\phi\in\mathcal{S}(\mathbb{R}^{n}),
\end{equation*}
\begin{equation*}
    B[u,v] := G\mathbb{P}\nabla\cdot(u\otimes v),
\end{equation*}
\begin{equation*}
    \widetilde{B}[u,v] := G{(\mathbb{P}\nabla\cdot)}_{\varphi}\mathsf{Bony}(u,v),
\end{equation*}
and we have the following lemma.
\begin{lemma}\label{bilinear-estimates-lemma}
    \begin{enumerate}[label=(\roman*)]
        \item
        For $T\in(0,\infty)$, and $(\alpha,\ell,\epsilon)\in{[1,\infty]}^{2}\times(0,1]$ satisfying \eqref{ale-conditions}, we consider the path space
        \begin{equation*}
            X_{\ell,\epsilon}^{\alpha}(T) = X_{\ell,\epsilon}^{\alpha}(T,n) := \left\{\begin{array}{ll} L_{T}^{\alpha}L^{\ell}(\mathbb{R}^{n}) & \text{if }s_{\ell}+\epsilon+\frac{2}{\alpha}=0, \\ \widetilde{L}_{T}^{\alpha}\dot{B}_{\ell,\infty}^{s_{\ell}+\epsilon+\frac{2}{\alpha}}(\mathbb{R}^{n}) & \text{if } s_{\ell}+\epsilon+\frac{2}{\alpha}>0. \end{array}\right.
        \end{equation*}
        Let $p_{\alpha,\epsilon}=n{\left(1-\epsilon-\frac{2}{\alpha}\right)}^{-1}$ (so that $s_{p_{\alpha,\epsilon}}+\epsilon+\frac{2}{\alpha}=0$). If $u$ and $v$ are (equivalence classes of) weakly* measurable functions $(0,T)\rightarrow\mathcal{S}'(\mathbb{R}^{n})$ with finite $X_{\ell,\epsilon}^{\alpha}(T)$-seminorm, then the expression
        \begin{equation*}
            B_{\alpha,\ell,\epsilon}[u,v] := \left\{\begin{array}{ll} B[u,v] & \text{if }\ell=p_{\alpha,\epsilon}, \\ \widetilde{B}[u,v] & \text{if } \ell<p_{\alpha,\epsilon}, \end{array}\right.
        \end{equation*}
        defines a weakly* measurable function $B_{\alpha,\ell,\epsilon}[u,v]:(0,T)\rightarrow\mathcal{S}_{h}'(\mathbb{R}^{n})$. Using the notation $\tau_{t_{0}}w(t)=w(t+t_{0})$, for all $0<t_{0}<t<T$ we have the semigroup property
        \begin{equation}\label{bilinear-semigroup}
            \tau_{t_{0}}B_{\alpha,\ell,\epsilon}[u,v](t) = S[B_{\alpha,\ell,\epsilon}[u,v](t_{0})](t) + B_{\alpha,\ell,\epsilon}[\tau_{t_{0}}u,\tau_{t_{0}}v](t) \quad \text{in }\mathcal{S}'(\mathbb{R}^{n}).
        \end{equation}
        We have the estimate\footnote{As pointed out in the proof of this lemma, the intermediate quanitity in \eqref{besov-uniqueness-bilinear-estimate} can be replaced by ${\|B_{\alpha,\ell,\epsilon}[u,v]\|}_{\widetilde{L}_{T}^{\frac{\alpha}{2}}\dot{B}_{\ell,1}^{s_{\ell}+\epsilon+\frac{4}{\alpha}}(\mathbb{R}^{n})\cap\widetilde{L}_{T}^{\overline{\infty}}\dot{B}_{\ell,1}^{s_{\ell}+\epsilon}(\mathbb{R}^{n})}$. For the purposes of this paper, we only need \eqref{besov-uniqueness-bilinear-estimate} as stated here.}
        \begin{equation}\label{besov-uniqueness-bilinear-estimate}
        \begin{aligned}
            {\|B_{\alpha,\ell,\epsilon}[u,v]\|}_{X_{\ell,\epsilon}^{\alpha}(T)} &\leq {\|B_{\alpha,\ell,\epsilon}[u,v]\|}_{\widetilde{L}_{T}^{\alpha}\dot{B}_{\ell,1}^{s_{\ell}+\epsilon+\frac{2}{\alpha}}(\mathbb{R}^{n})\cap\widetilde{L}_{T}^{\overline{\infty}}\dot{B}_{\ell,1}^{s_{\ell}+\epsilon}(\mathbb{R}^{n})} \\
            &\lesssim_{\varphi,\alpha,\ell,\epsilon} T^{\frac{\epsilon}{2}}{\|u\|}_{X_{\ell,\epsilon}^{\alpha}(T)}{\|v\|}_{X_{\ell,\epsilon}^{\alpha}(T)}.
        \end{aligned}
        \end{equation}
        \item
        For $T\in(0,\infty)$, $\epsilon\in(0,1)$, $s\in(-1,\infty)$ and $p,q\in[1,\infty]$, we consider the path spaces
        \begin{equation*}
            Y_{\epsilon/2}(T) = Y_{\epsilon/2}(T,n) := \widetilde{L}_{T}^{\frac{4}{2+\epsilon}}\dot{B}_{\infty,\infty}^{\epsilon}(\mathbb{R}^{n}) \cap L_{T}^{\frac{4}{2-\epsilon}}L^{\infty}(\mathbb{R}^{n}),
        \end{equation*}
        \begin{equation*}
            Z_{p,q}^{s}(T) = Z_{p,q}^{s}(T,n) := \widetilde{L}_{T}^{1}\dot{B}_{p,q}^{s+2}(\mathbb{R}^{n})\cap\widetilde{L}_{T}^{\overline{\infty}}\dot{B}_{p,q}^{s}(\mathbb{R}^{n}),
        \end{equation*}
        which satisfy the embeddings ${\|u\|}_{X_{\infty,\epsilon/2}^{4/(2-\epsilon)}(T)}\leq{\|u\|}_{Y_{\epsilon/2}(T)}\lesssim_{\epsilon}T^{\epsilon/4}{\|u\|}_{Z_{\infty,\infty}^{-1+\epsilon}(T)}$. If $u$ and $v$ belong to the space $Z_{\infty,\infty}^{-1+\epsilon}(T)$, then $B[u,v]$ and $\widetilde{B}[u,v]$ are well-defined and coincide with each other, and satisfy the estimate
        \begin{equation}\label{besov-existence-bilinear-estimate}
            {\|B[u,v]\|}_{Z_{\infty,\infty}^{-1+\epsilon}(T)} \lesssim_{\varphi} \frac{1}{\epsilon}\left({\|u\|}_{Y_{\epsilon/2}(T)}{\|v\|}_{Y_{\epsilon/2}(T)}\wedge\frac{1}{1-\epsilon}T^{\frac{\epsilon}{2}}{\|u\|}_{Z_{\infty,\infty}^{-1+\epsilon}(T)}{\|v\|}_{Z_{\infty,\infty}^{-1+\epsilon}(T)}\right).
        \end{equation}
        Choosing $\lambda\in(0,1)$ such that $\lambda(s+1-\epsilon)<1-\epsilon$, we have the further estimate
        \begin{equation}\label{persistence-bilinear-estimate}
        \begin{aligned}
            &\quad{\|B[u,v]\|}_{Z_{p,q}^{s}(T)} + {\|B[v,u]\|}_{Z_{p,q}^{s}(T)} \\
            &\lesssim_{\varphi,s,\epsilon,\lambda} T^{\frac{\epsilon}{2}}\left({\|u\|}_{Z_{\infty,\infty}^{-1+\epsilon}(T)}{\|v\|}_{Z_{p,q}^{s}(T)}+{\|u\|}_{Z_{\infty,\infty}^{-1+\epsilon}(T)}^{\lambda}{\|u\|}_{Z_{p,q}^{s}(T)}^{1-\lambda}{\|v\|}_{Z_{\infty,\infty}^{-1+\epsilon}(T)}^{1-\lambda}{\|v\|}_{Z_{p,q}^{s}(T)}^{\lambda}\right).
        \end{aligned}
        \end{equation}
    \end{enumerate}
\end{lemma}
\begin{remark}\label{bilinear-estimates-remark}
\begin{enumerate}[label=(\roman*)]
    \item
    In the case $\ell=p_{\alpha,\epsilon}$, the estimate \eqref{besov-uniqueness-bilinear-estimate} (without the intermediate quantity) was proved in \cite{fabes1972}. A version in Lorentz spaces is proved in \cite{davies-lorentz-blowup}.
    \item
    The estimate ${\|B[u,v]\|}_{Z_{\infty,\infty}^{-1+\epsilon}(T)}\lesssim_{\varphi,\epsilon}T^{\epsilon/2}{\|u\|}_{Z_{\infty,\infty}^{-1+\epsilon}(T)}{\|v\|}_{Z_{\infty,\infty}^{-1+\epsilon}(T)}$ was proved in \cite[Theorem 1.3]{chemin2019a}. Our estimate \eqref{besov-existence-bilinear-estimate} provides information on how the implied constant depends on $\epsilon$, and also establishes the stronger estimate ${\|B[u,v]\|}_{Z_{\infty,\infty}^{-1+\epsilon}(T)}\lesssim_{\varphi,\epsilon}{\|u\|}_{Y_{\epsilon/2}(T)}{\|v\|}_{Y_{\epsilon/2}(T)}$.
    \item
    The estimate \eqref{persistence-bilinear-estimate}, and its application in proving Theorem \ref{besov-local-existence-theorem}, are based on the ideas of \cite[Theorem 9.11]{lemarie2016}, in which the space $Z_{\infty,\infty}^{-1+\epsilon}(T)$ (naturally arising from the heat flow of an element of $\dot{B}_{\infty,\infty}^{-1+\epsilon}(\mathbb{R}^{n})$) is replaced by a space which naturally arises from the heat flow of an element of ${\mathrm{BMO}}^{-1}(\mathbb{R}^{n})$.
\end{enumerate}
\end{remark}
\begin{proof}[Proof of Lemma \ref{bilinear-estimates-lemma}]
    We start by noting some general properties of $G$. The semigroup property \eqref{bilinear-semigroup} follows from the calculation
    \begin{equation*}
    \begin{aligned}
        \langle G[w](t+t_{0}),\phi\rangle &= \int_{0}^{t+t_{0}}\langle e^{(t+t_{0}-s)\Delta}w(s),\phi\rangle\,\mathrm{d}s \\
        &= \int_{0}^{t_{0}}\langle e^{(t_{0}-s)\Delta}w(s),e^{t\Delta}\phi\rangle\,\mathrm{d}s + \int_{t_{0}}^{t+t_{0}}\langle e^{(t+t_{0}-s)\Delta}w(s),\phi\rangle\,\mathrm{d}s \\
        &= \langle G[w](t_{0}),e^{t\Delta}\phi\rangle + \int_{0}^{t}\langle e^{(t-\widetilde{s})\Delta}w(\widetilde{s}+t_{0}),\phi\rangle\,\mathrm{d}\widetilde{s} \\
        &= \langle S[G[w](t_{0})](t),\phi\rangle + \langle G[\tau_{t_{0}}w](t),\phi\rangle.
    \end{aligned}
    \end{equation*}

    To establish estimates for $G$, we define the expressions $G_{j}w(t,x) := \int_{0}^{t}[e^{(t-s)\Delta}\dot{\Delta}_{j}w(s)](x)\,\mathrm{d}s$ and $\widetilde{G}_{j,r}w(t) := \int_{0}^{t}{\|e^{(t-s)\Delta}\dot{\Delta}_{j}w(s)\|}_{L^{r}(\mathbb{R}^{n})}\,\mathrm{d}s$. By Minkowski's inequality and \eqref{kappa} we have the estimate ${\|G_{j}w(t)\|}_{L^{r}(\mathbb{R}^{n})}\leq\widetilde{G}_{j,r}w(t)\lesssim_{n}\int_{0}^{t}\kappa_{j}(t-s){\|\dot{\Delta}_{j}w(s)\|}_{L^{r}(\mathbb{R}^{n})}\,\mathrm{d}s$, so by \eqref{kappa} and Young's convolution inequality we have, for $1\leq\beta\leq\gamma\leq\infty$,
    \begin{equation}\label{Gj-estimate}
        {\|G_{j}w\|}_{L_{T}^{\gamma}L^{r}(\mathbb{R}^{n})} \leq {\|\widetilde{G}_{j,r}w\|}_{L_{T}^{\gamma}} \lesssim_{n} \left(2^{-2j\left(1-\frac{1}{\beta}+\frac{1}{\gamma}\right)}\wedge T^{1-\frac{1}{\beta}+\frac{1}{\gamma}}\right){\|\dot{\Delta}_{j}w\|}_{L_{T}^{\beta}L^{r}(\mathbb{R}^{n})},
    \end{equation}
    where we can replace $L_{T}^{\infty}$ by $L_{T}^{\overline{\infty}}$ in the case $\gamma=\infty$.

    If it is the case that the expression $\langle Gw(t),\phi\rangle:=\int_{0}^{t}\langle e^{(t-s)\Delta}w(s),\phi\rangle\,\mathrm{d}s$ defines a weakly* measurable function $Gw:(0,T)\rightarrow\mathcal{S}'(\mathbb{R}^{n})$ satisfying
    \begin{equation}\label{bilinear-G-projection}
        \langle\dot{\Delta}_{j}Gw(t),\phi\rangle=\int_{\mathbb{R}^{n}}G_{j}w(t,x)\phi(x)\,\mathrm{d}x \quad \text{for all }t\in(0,T)\text{ and }\phi\in\mathcal{S}(\mathbb{R}^{n}),
    \end{equation}
    then we can use \eqref{Gj-estimate} to estimate
    \begin{equation*}
        {\|Gw\|}_{\widetilde{L}_{T}^{\gamma}\dot{B}_{r,q}^{s+2\left(1-\frac{1}{\beta}+\frac{1}{\gamma}\right)}(\mathbb{R}^{n})} \lesssim_{n} {\|w\|}_{\widetilde{L}_{T}^{\beta}\dot{B}_{r,q}^{s}(\mathbb{R}^{n})}, \quad {\|Gw\|}_{\widetilde{L}_{T}^{\gamma}\dot{B}_{r,q}^{s}(\mathbb{R}^{n})} \lesssim_{n} T^{1-\frac{1}{\beta}+\frac{1}{\gamma}}{\|w\|}_{\widetilde{L}_{T}^{\beta}\dot{B}_{r,q}^{s}(\mathbb{R}^{n})},
    \end{equation*}
    so by \eqref{interpolation-holder} we obtain
    \begin{equation}\label{bilinear-G-estimate}
        {\|Gw\|}_{\widetilde{L}_{T}^{\gamma}\dot{B}_{r,q}^{s+2\delta}(\mathbb{R}^{n})} \lesssim_{n} T^{1-\frac{1}{\beta}+\frac{1}{\gamma}-\delta}{\|w\|}_{\widetilde{L}_{T}^{\beta}\dot{B}_{r,q}^{s}(\mathbb{R}^{n})} \quad \text{for }\beta\leq\gamma\text{ and }0\leq\delta\leq1-\frac{1}{\beta}+\frac{1}{\gamma},
    \end{equation}
    where we make the following remarks:
    \begin{enumerate}[label=(\alph*)]
        \item If $\gamma=\infty$, then we can replace $\widetilde{L}_{T}^{\infty}$ with $\widetilde{L}_{T}^{\overline{\infty}}$ on the left hand side of \eqref{bilinear-G-estimate}.
        \item If $0<\delta<1-\frac{1}{\beta}+\frac{1}{\gamma}$, then we can use \eqref{interpolation-geometric} instead of \eqref{interpolation-holder} to replace $\dot{B}_{r,q}^{s+2\delta}$ with $\dot{B}_{r,1}^{s+2\delta}$ on the left hand side of \eqref{bilinear-G-estimate}, with the implied constant now depending on $n,s,\beta,\gamma,\delta$.
    \end{enumerate}

    The following lemma (which we prove in the appendix) provides a means of justifying \eqref{bilinear-G-estimate}.
    \begin{lemma}\label{bilinear-G-lemma}
        If $w\in\widetilde{L}_{T}^{\beta}\dot{B}_{r,q}^{s}(\mathbb{R}^{n})$ for some $(s,r,q)$ satisfying \eqref{negative-scaling}, then $Gw$ does indeed define a weakly* measurable function $Gw:(0,T)\rightarrow\mathcal{S}'(\mathbb{R}^{n})$ satisfying \eqref{bilinear-G-projection}.
    \end{lemma}
    In the setting of Lemma \ref{bilinear-G-lemma}, the estimate \eqref{bilinear-G-estimate} is justified. Noting the particular case $(\gamma,\delta)=(\infty,0)$, it follows from Lemma \ref{convergence-lemma} that $Gw(t)\in\mathcal{S}_{h}'(\mathbb{R}^{n})$ for all $t\in(0,T)$.

    We now consider the two parts of Lemma \ref{bilinear-estimates-lemma} as follows.
    \begin{enumerate}[label=(\roman*)]
        \item
        To simplify the notation, let $p=p_{\alpha,\epsilon}$.

        In the case $\ell=p$, the identity $\dot{\Delta}_{j}\mathbb{P}\nabla\cdot(u\otimes v)=\mathbb{P}\nabla\cdot\dot{\Delta}_{j}(u\otimes v)$ is justified by Lemma \ref{sigma-lemma} from the appendix, so we can estimate
        \begin{equation*}\label{B-product-estimate-1}
        \begin{aligned}
            {\|\mathbb{P}\nabla\cdot(u\otimes v)\|}_{\widetilde{L}_{T}^{\frac{\alpha}{2}}\dot{B}_{p,\infty}^{-1-\frac{n}{p}}(\mathbb{R}^{n})} \lesssim_{n} {\|u\otimes v\|}_{\widetilde{L}_{T}^{\frac{\alpha}{2}}\dot{B}_{p,\infty}^{-\frac{n}{p}}(\mathbb{R}^{n})} &\lesssim_{\varphi} {\|u\otimes v\|}_{L_{T}^{\frac{\alpha}{2}}L^{\frac{p}{2}}(\mathbb{R}^{n})} \\
            &\leq {\|u\|}_{L_{T}^{\alpha}L^{p}(\mathbb{R}^{n})}{\|v\|}_{L_{T}^{\alpha}L^{p}(\mathbb{R}^{n})},
        \end{aligned}
        \end{equation*}
        where we used \eqref{useful-inequality-2} and the identity $\dot{\Delta}_{j}\mathbb{P}\nabla\cdot(u\otimes v)=\mathbb{P}\nabla\cdot\dot{\Delta}_{j}(u\otimes v)$ for the first inequality, we used the embeddings \eqref{besov-embedding}-\eqref{rough-lp} for the second inequality, and we used H\"{o}lder's inequality for the third inequality.

        In the case $\ell<p$, we have $p\in(2,\infty)$ and $\ell\in[1,p)$, so we can choose $q\in[\ell',\infty]$ satisfying $\frac{1}{p}>\frac{1}{q}>\frac{2}{p}-\frac{1}{\ell}$, and we can estimate
        \begin{equation}\label{B-product-estimate-2}
        \begin{aligned}
            &\quad {\|{(\mathbb{P}\nabla\cdot)}_{\varphi}\mathsf{Bony}(u,v)\|}_{\widetilde{L}_{T}^{\frac{\alpha}{2}}\dot{B}_{\ell,\infty}^{-1+\frac{n}{\ell}-\frac{2n}{p}}(\mathbb{R}^{n})} \\
            &\lesssim_{\varphi,\ell,p} {\|\mathsf{Bony}(u,v)\|}_{\widetilde{L}_{T}^{\frac{\alpha}{2}}\dot{B}_{\ell,\infty}^{\frac{n}{\ell}-\frac{2n}{p}}(\mathbb{R}^{n})} \\
            &\lesssim_{n} {\|\mathsf{Bony}(u,v)\|}_{\widetilde{L}_{T}^{\frac{\alpha}{2}}\dot{B}_{\frac{\ell q}{\ell+q},\infty}^{\frac{n}{\ell}+\frac{n}{q}-\frac{2n}{p}}(\mathbb{R}^{n})} \\
            &\lesssim_{\varphi,\ell,p,q} {\|u\|}_{\widetilde{L}_{T}^{\alpha}\dot{B}_{q,\infty}^{\frac{n}{q}-\frac{n}{p}}(\mathbb{R}^{n})}{\|v\|}_{\widetilde{L}_{T}^{\alpha}\dot{B}_{\ell,\infty}^{\frac{n}{\ell}-\frac{n}{p}}(\mathbb{R}^{n})} + {\|u\|}_{\widetilde{L}_{T}^{\alpha}\dot{B}_{\ell,\infty}^{\frac{n}{\ell}-\frac{n}{p}}(\mathbb{R}^{n})}{\|v\|}_{\widetilde{L}_{T}^{\alpha}\dot{B}_{q,\infty}^{\frac{n}{q}-\frac{n}{p}}(\mathbb{R}^{n})} \\
            &\lesssim_{n} {\|u\|}_{\widetilde{L}_{T}^{\alpha}\dot{B}_{\ell,\infty}^{\frac{n}{\ell}-\frac{n}{p}}(\mathbb{R}^{n})}{\|v\|}_{\widetilde{L}_{T}^{\alpha}\dot{B}_{\ell,\infty}^{\frac{n}{\ell}-\frac{n}{p}}(\mathbb{R}^{n})},
        \end{aligned}
        \end{equation}
        where we used \eqref{useful-inequality-2} and Lemma \ref{convergence-lemma} for the first inequality, we used the embedding \eqref{besov-embedding} for the second and fourth inequalities, and we used the Bony estimates \eqref{bony-estimate-1} and \eqref{bony-estimate-3} for the third inequality. Convergence and measurability issues are addressed in Lemma \ref{product-lemma} and Lemma \ref{sigma-lemma} from the appendix.

        In the case $\ell=p$, we have $-1-\frac{n}{p}=-1+\frac{n}{\ell}-\frac{2n}{p}$. In all cases ($\ell=p$ or $\ell<p$), we have $-1+\frac{n}{\ell}-\frac{2n}{p}=s_{\ell}+2\epsilon+\frac{4}{\alpha}-2$, so to derive \eqref{besov-uniqueness-bilinear-estimate} it remains for us to establish the estimate
        \begin{equation}\label{uni-bilinear-G-estimate}
            {\|Gw\|}_{\widetilde{L}_{T}^{\alpha}\dot{B}_{\ell,1}^{s_{\ell}+\epsilon+\frac{2}{\alpha}}(\mathbb{R}^{n})\cap\widetilde{L}_{T}^{\overline{\infty}}\dot{B}_{\ell,1}^{s_{\ell}+\epsilon}(\mathbb{R}^{n})} \lesssim_{n,\alpha,\ell,\epsilon} T^{\frac{\epsilon}{2}}{\|w\|}_{\widetilde{L}_{T}^{\frac{\alpha}{2}}\dot{B}_{\ell,\infty}^{s_{\ell}+2\epsilon+\frac{4}{\alpha}-2}(\mathbb{R}^{n})}.
        \end{equation}
        The estimate \eqref{uni-bilinear-G-estimate} is established\footnote{In fact we establish a stronger estimate, in which the left hand side of \eqref{uni-bilinear-G-estimate} is replaced by ${\|Gw\|}_{\widetilde{L}_{T}^{\frac{\alpha}{2}}\dot{B}_{\ell,1}^{s_{\ell}+\epsilon+\frac{4}{\alpha}}(\mathbb{R}^{n})\cap\widetilde{L}_{T}^{\overline{\infty}}\dot{B}_{\ell,1}^{s_{\ell}+\epsilon}(\mathbb{R}^{n})}$. By the interpolation \eqref{interpolation-holder} we have $\widetilde{L}_{T}^{\frac{\alpha}{2}}\dot{B}_{\ell,1}^{s_{\ell}+\epsilon+\frac{4}{\alpha}}(\mathbb{R}^{n})\cap\widetilde{L}_{T}^{\overline{\infty}}\dot{B}_{\ell,1}^{s_{\ell}+\epsilon}(\mathbb{R}^{n})\subseteq\widetilde{L}_{T}^{\alpha}\dot{B}_{\ell,1}^{s_{\ell}+\epsilon+\frac{2}{\alpha}}(\mathbb{R}^{n})$, so the footnote associated with \eqref{besov-uniqueness-bilinear-estimate} is established.} by applying Lemma \ref{bilinear-G-lemma} and the estimate \eqref{bilinear-G-estimate} (with the related remarks (a) and (b)), with $r=\ell$, $q=\infty$, $s=s_{\ell}+2\epsilon+\frac{4}{\alpha}-2$, $\beta=\frac{\alpha}{2}$, $\gamma\in[\beta,\infty]$, and $\delta=1-\frac{1}{\beta}+\frac{1}{\gamma}-\frac{\epsilon}{2}$, so that
        \begin{equation*}
            s+2\delta=s_{\ell}+\epsilon+\frac{2}{\gamma}, \quad \delta\geq1-\frac{2}{\alpha}-\frac{\epsilon}{2}>1-\frac{2}{\alpha}-\epsilon\geq0, \quad 1-\frac{1}{\beta}+\frac{1}{\gamma}-\delta=\frac{\epsilon}{2}>0;
        \end{equation*}
        the inequality $\epsilon+\frac{2}{\alpha}\leq1$ ensures that $(s,r,q)$ satisfy \eqref{negative-scaling}; so the use of Lemma \ref{bilinear-G-lemma} is justified.
        \item
        We have the embeddings
        \begin{equation*}
        \begin{aligned}
            {\|u\|}_{L_{T}^{4/(2-\epsilon)}L^{\infty}(\mathbb{R}^{n})} &\leq {\|u\|}_{\widetilde{L}_{T}^{4/(2+\epsilon)}\dot{B}_{\infty,\infty}^{\epsilon}(\mathbb{R}^{n})} + {\|u\|}_{\widetilde{L}_{T}^{4/(2-\epsilon)}\dot{B}_{\infty,1}^{0}(\mathbb{R}^{n})} \\
            &\lesssim_{\epsilon} T^{\frac{\epsilon}{4}}{\|u\|}_{\widetilde{L}_{T}^{2}\dot{B}_{\infty,\infty}^{\epsilon}(\mathbb{R}^{n})} + {\|u\|}_{\widetilde{L}_{T}^{4/(2-\epsilon)}\dot{B}_{\infty,\infty}^{\epsilon/2}(\mathbb{R}^{n})}^{(2-2\epsilon)/(2-\epsilon)}{\|u\|}_{\widetilde{L}_{T}^{4/(2-\epsilon)}\dot{B}_{\infty,\infty}^{-1+\epsilon}(\mathbb{R}^{n})}^{\epsilon/(2-\epsilon)} \\
            &\leq T^{\frac{\epsilon}{4}}{\|u\|}_{\widetilde{L}_{T}^{2}\dot{B}_{\infty,\infty}^{\epsilon}(\mathbb{R}^{n})} + {\|u\|}_{\widetilde{L}_{T}^{4/(2-\epsilon)}\dot{B}_{\infty,\infty}^{\epsilon/2}(\mathbb{R}^{n})}^{(2-2\epsilon)/(2-\epsilon)}\left(T^{\frac{2-\epsilon}{4}}{\|u\|}_{\widetilde{L}_{T}^{\infty}\dot{B}_{\infty,\infty}^{-1+\epsilon}(\mathbb{R}^{n})}\right)^{\epsilon/(2-\epsilon)} \\
            &\lesssim T^{\frac{\epsilon}{4}}{\|u\|}_{\widetilde{L}_{T}^{1}\dot{B}_{\infty,\infty}^{1+\epsilon}(\mathbb{R}^{n})\cap\widetilde{L}_{T}^{\infty}\dot{B}_{\infty,\infty}^{-1+\epsilon}(\mathbb{R}^{n})},
        \end{aligned}
        \end{equation*}
        where we used \eqref{interpolation-geometric} for the second inequality, and we used \eqref{interpolation-holder} for the fourth inequality. We deduce that the embedding ${\|u\|}_{X_{\infty,\epsilon/2}^{4/(2-\epsilon)}(T)}\leq{\|u\|}_{Y_{\epsilon/2}(T)}\lesssim_{\epsilon}T^{\epsilon/4}{\|u\|}_{Z_{\infty,\infty}^{-1+\epsilon}(T)}$ holds, and that $Z_{\infty,\infty}^{-1+\epsilon}(T) \subseteq \widetilde{L}_{T}^{2}\dot{B}_{\infty,1}^{0}(\mathbb{R}^{n})\cap\bigcup_{\alpha\in(2,\infty)}L_{T}^{\alpha}L^{\infty}(\mathbb{R}^{n})$. By the remarks preceding Lemma \ref{paraproduct} about convergence of the Bony product, and by Lemma \ref{sigma-lemma}, the fact that $u,v\in\widetilde{L}_{T}^{2}\dot{B}_{\infty,1}^{0}(\mathbb{R}^{n})$ ensures that the identities $u\otimes v=\mathsf{Bony}(u,v)$, $\mathbb{P}\nabla\cdot(u\otimes v)={(\mathbb{P}\nabla\cdot)}_{\varphi}(u\otimes v)$ and $\dot{\Delta}_{j}\mathbb{P}\nabla\cdot(u\otimes v)=\mathbb{P}\nabla\cdot\dot{\Delta}_{j}(u\otimes v)$ are satisfied (and that $B[u,v]=\widetilde{B}[u,v]$ if at least one of them is defined). By part (i) of this lemma, the fact that $u,v\in\bigcup_{\alpha\in(2,\infty)}L_{T}^{\alpha}L^{\infty}(\mathbb{R}^{n})$ ensures that $B[u,v]$ is well-defined, and that estimates of the form \eqref{bilinear-G-estimate} are justified when $w=\mathbb{P}\nabla\cdot(u\otimes v)$.

        We note the general estimate
        \begin{equation}\label{exi-bilinear-G-estimate}
        \begin{aligned}
            {\|B[u,v]\|}_{Z_{r,q}^{s}(T)} \lesssim_{n} {\|\mathbb{P}\nabla\cdot(u\otimes v)\|}_{\widetilde{L}_{T}^{1}\dot{B}_{r,q}^{s}(\mathbb{R}^{n})} &\lesssim_{n} {\|u\otimes v\|}_{\widetilde{L}_{T}^{1}\dot{B}_{r,q}^{s+1}(\mathbb{R}^{n})} \\
            &\leq T^{\frac{\epsilon}{2}}{\|u\otimes v\|}_{\widetilde{L}_{T}^{\frac{2}{2-\epsilon}}\dot{B}_{r,q}^{s+1}(\mathbb{R}^{n})},
        \end{aligned}
        \end{equation}
        where we used \eqref{useful-inequality-2} and the identity $\dot{\Delta}_{j}\mathbb{P}\nabla\cdot(u\otimes v)=\mathbb{P}\nabla\cdot\dot{\Delta}_{j}(u\otimes v)$ for the second inequality; for the first inequality, we used \eqref{bilinear-G-estimate} in the cases $(\beta,\gamma,\delta)=(1,1,1)$ and $(\beta,\gamma,\delta)=(1,\infty,0)$.

        The estimate ${\|B[u,v]\|}_{Z_{\infty,\infty}^{-1+\epsilon}(T)}\lesssim_{\varphi}\frac{1}{\epsilon}{\|u\|}_{Y_{\epsilon/2}(T)}{\|v\|}_{Y_{\epsilon/2}(T)}$ follows from the calculation
        \begin{equation*}
        \begin{aligned}
            &\quad {\|B[u,v]\|}_{Z_{\infty,\infty}^{-1+\epsilon}(T)} \lesssim_{n} {\|u\otimes v\|}_{\widetilde{L}_{T}^{1}\dot{B}_{\infty,\infty}^{\epsilon}(\mathbb{R}^{n})} \\
            &\lesssim_{\varphi} \frac{1}{\epsilon}\left({\|u\|}_{L_{T}^{\frac{4}{2-\epsilon}}L^{\infty}(\mathbb{R}^{n})}{\|v\|}_{\widetilde{L}_{T}^{\frac{4}{2+\epsilon}}\dot{B}_{\infty,\infty}^{\epsilon}(\mathbb{R}^{n})} + {\|u\|}_{\widetilde{L}_{T}^{\frac{4}{2+\epsilon}}\dot{B}_{\infty,\infty}^{\epsilon}(\mathbb{R}^{n})}{\|v\|}_{L_{T}^{\frac{4}{2-\epsilon}}L^{\infty}(\mathbb{R}^{n})}\right) \\
            &\leq \frac{1}{\epsilon}{\|u\|}_{Y_{\epsilon/2}(T)}{\|v\|}_{Y_{\epsilon/2}(T)},
        \end{aligned}
        \end{equation*}
        where we used \eqref{exi-bilinear-G-estimate} for the first inequality; for the second inequality, we used the identity $u\otimes v=\mathsf{Bony}(u,v)$, together with the Bony estimates
        \begin{equation*}
            {\|\dot{T}_{u}v\|}_{\widetilde{L}_{T}^{1}\dot{B}_{\infty,\infty}^{\epsilon}(\mathbb{R}^{n})} \lesssim_{\varphi} {\|u\|}_{L_{T}^{\frac{4}{2-\epsilon}}L^{\infty}(\mathbb{R}^{n})}{\|v\|}_{\widetilde{L}_{T}^{\frac{4}{2+\epsilon}}\dot{B}_{\infty,\infty}^{\epsilon}(\mathbb{R}^{n})},
        \end{equation*}
        \begin{equation*}
            {\|\dot{T}_{v}u\|}_{\widetilde{L}_{T}^{1}\dot{B}_{\infty,\infty}^{\epsilon}(\mathbb{R}^{n})} \lesssim_{\varphi} {\|v\|}_{L_{T}^{\frac{4}{2-\epsilon}}L^{\infty}(\mathbb{R}^{n})}{\|u\|}_{\widetilde{L}_{T}^{\frac{4}{2+\epsilon}}\dot{B}_{\infty,\infty}^{\epsilon}(\mathbb{R}^{n})},
        \end{equation*}
        \begin{equation*}
            {\|\dot{R}(u,v)\|}_{\widetilde{L}_{T}^{1}\dot{B}_{\infty,\infty}^{\epsilon}(\mathbb{R}^{n})} \lesssim_{\varphi} \frac{1}{\epsilon}{\|u\|}_{L_{T}^{\frac{4}{2-\epsilon}}\dot{B}_{\infty,\infty}^{0}(\mathbb{R}^{n})}{\|v\|}_{\widetilde{L}_{T}^{\frac{4}{2+\epsilon}}\dot{B}_{\infty,\infty}^{\epsilon}(\mathbb{R}^{n})}.
        \end{equation*}
        The estimate ${\|B[u,v]\|}_{Z_{\infty,\infty}^{-1+\epsilon}(T)}\lesssim_{\varphi}\frac{1}{\epsilon(1-\epsilon)}T^{\epsilon/2}{\|u\|}_{Z_{\infty,\infty}^{-1+\epsilon}(T)}{\|v\|}_{Z_{\infty,\infty}^{-1+\epsilon}(T)}$ follows from the calculation
        \begin{equation*}
        \begin{aligned}
            &\quad {\|B[u,v]\|}_{Z_{\infty,\infty}^{-1+\epsilon}(T)} \lesssim_{n} T^{\frac{\epsilon}{2}}{\|u\otimes v\|}_{\widetilde{L}_{T}^{\frac{2}{2-\epsilon}}\dot{B}_{\infty,\infty}^{\epsilon}(\mathbb{R}^{n})} \\
            &\lesssim_{\varphi} \frac{1}{\epsilon(1-\epsilon)}T^{\frac{\epsilon}{2}}\left({\|u\|}_{\widetilde{L}_{T}^{\infty}\dot{B}_{\infty,\infty}^{-1+\epsilon}(T)}{\|v\|}_{\widetilde{L}_{T}^{\frac{2}{2-\epsilon}}\dot{B}_{\infty,\infty}^{1}(\mathbb{R}^{n})} + {\|u\|}_{\widetilde{L}_{T}^{\frac{2}{2-\epsilon}}\dot{B}_{\infty,\infty}^{1}(\mathbb{R}^{n})}{\|v\|}_{\widetilde{L}_{T}^{\infty}\dot{B}_{\infty,\infty}^{-1+\epsilon}(T)}\right) \\
            &\leq \frac{1}{\epsilon(1-\epsilon)}T^{\frac{\epsilon}{2}}{\|u\|}_{Z_{\infty,\infty}^{-1+\epsilon}(T)}{\|v\|}_{Z_{\infty,\infty}^{-1+\epsilon}(T)},
        \end{aligned}
        \end{equation*}
        where we used \eqref{exi-bilinear-G-estimate} for the first inequality, and we used the interpolation \eqref{interpolation-holder} for the third inequality; for the second inequality, we used the identity $u\otimes v=\mathsf{Bony}(u,v)$, together with the Bony estimates
        \begin{equation*}
            {\|\dot{T}_{u}v\|}_{\widetilde{L}_{T}^{\frac{2}{2-\epsilon}}\dot{B}_{\infty,\infty}^{\epsilon}(\mathbb{R}^{n})} \lesssim_{\varphi} \frac{1}{1-\epsilon}{\|u\|}_{\widetilde{L}_{T}^{\infty}\dot{B}_{\infty,\infty}^{-1+\epsilon}(T)}{\|v\|}_{\widetilde{L}_{T}^{\frac{2}{2-\epsilon}}\dot{B}_{\infty,\infty}^{1}(\mathbb{R}^{n})},
        \end{equation*}
        \begin{equation*}
            {\|\dot{T}_{v}u\|}_{\widetilde{L}_{T}^{\frac{2}{2-\epsilon}}\dot{B}_{\infty,\infty}^{\epsilon}(\mathbb{R}^{n})} \lesssim_{\varphi} \frac{1}{1-\epsilon}{\|v\|}_{\widetilde{L}_{T}^{\infty}\dot{B}_{\infty,\infty}^{-1+\epsilon}(T)}{\|u\|}_{\widetilde{L}_{T}^{\frac{2}{2-\epsilon}}\dot{B}_{\infty,\infty}^{1}(\mathbb{R}^{n})},
        \end{equation*}
        \begin{equation*}
            {\|\dot{R}(u,v)\|}_{\widetilde{L}_{T}^{\frac{2}{2-\epsilon}}\dot{B}_{\infty,\infty}^{\epsilon}(\mathbb{R}^{n})} \lesssim_{\varphi} \frac{1}{\epsilon}{\|u\|}_{\widetilde{L}_{T}^{\infty}\dot{B}_{\infty,\infty}^{-1+\epsilon}(T)}{\|v\|}_{\widetilde{L}_{T}^{\frac{2}{2-\epsilon}}\dot{B}_{\infty,\infty}^{1}(\mathbb{R}^{n})}.
        \end{equation*}
        We have establised the estimate ${\|B[u,v]\|}_{Z_{\infty,\infty}^{-1+\epsilon}(T)}\lesssim_{\varphi}\frac{1}{\epsilon}{\|u\|}_{Y_{\epsilon/2}(T)}{\|v\|}_{Y_{\epsilon/2}(T)}$ and the estimate ${\|B[u,v]\|}_{Z_{\infty,\infty}^{-1+\epsilon}(T)}\lesssim_{\varphi}\frac{1}{\epsilon(1-\epsilon)}T^{\epsilon/2}{\|u\|}_{Z_{\infty,\infty}^{-1+\epsilon}(T)}{\|v\|}_{Z_{\infty,\infty}^{-1+\epsilon}(T)}$, so the derivation of \eqref{besov-existence-bilinear-estimate} is complete.

        The estimate for ${\|B[u,v]\|}_{Z_{p,q}^{s}(T)}$ in \eqref{persistence-bilinear-estimate} follows from the calculation
        \begin{equation*}
        \begin{aligned}
            &\quad {\|B[u,v]\|}_{Z_{p,q}^{s}(T)} \lesssim_{n} T^{\frac{\epsilon}{2}}{\|u\otimes v\|}_{\widetilde{L}_{T}^{\frac{2}{2-\epsilon}}\dot{B}_{p,q}^{s+1}(\mathbb{R}^{n})} \\
            &\lesssim_{\varphi,s,\epsilon,\lambda}  T^{\frac{\epsilon}{2}}\left({\|u\|}_{\widetilde{L}_{T}^{\infty}\dot{B}_{\infty,\infty}^{-1+\epsilon}(\mathbb{R}^{n})}{\|v\|}_{\widetilde{L}_{T}^{\frac{2}{2-\epsilon}}\dot{B}_{p,q}^{s+2-\epsilon}(\mathbb{R}^{n})}\right. \\
            &\qquad\qquad\qquad\qquad\left. + {\|u\|}_{\widetilde{L}_{T}^{\frac{2}{2-\epsilon}}\dot{B}_{\frac{p}{1-\lambda},\frac{q}{1-\lambda}}^{(1-\lambda)(s+2-\epsilon)+\lambda}(\mathbb{R}^{n})}{\|v\|}_{\widetilde{L}_{T}^{\infty}\dot{B}_{\frac{p}{\lambda},\frac{q}{\lambda}}^{\lambda s+(1-\lambda)(-1+\epsilon)}(\mathbb{R}^{n})}\right) \\
            &\leq T^{\frac{\epsilon}{2}}\left({\|u\|}_{\widetilde{L}_{T}^{\infty}\dot{B}_{\infty,\infty}^{-1+\epsilon}(\mathbb{R}^{n})}{\|v\|}_{\widetilde{L}_{T}^{\frac{2}{2-\epsilon}}\dot{B}_{p,q}^{s+2-\epsilon}(\mathbb{R}^{n})} \right. \\
            &\qquad\qquad\qquad\qquad\left. + {\|u\|}_{\widetilde{L}_{T}^{\frac{2}{2-\epsilon}}\dot{B}_{p,q}^{s+2-\epsilon}(\mathbb{R}^{n})}^{1-\lambda}{\|u\|}_{\widetilde{L}_{T}^{\frac{2}{2-\epsilon}}\dot{B}_{\infty,\infty}^{1}(\mathbb{R}^{n})}^{\lambda}{\|v\|}_{\widetilde{L}_{T}^{\infty}\dot{B}_{p,q}^{s}(\mathbb{R}^{n})}^{\lambda}{\|v\|}_{\widetilde{L}_{T}^{\infty}\dot{B}_{\infty,\infty}^{-1+\epsilon}(\mathbb{R}^{n})}^{1-\lambda}\right) \\
            &\leq T^{\frac{\epsilon}{2}}\left({\|u\|}_{Z_{\infty,\infty}^{-1+\epsilon}(T)}{\|v\|}_{Z_{p,q}^{s}(T)}+{\|u\|}_{Z_{\infty,\infty}^{-1+\epsilon}(T)}^{\lambda}{\|u\|}_{Z_{p,q}^{s}(T)}^{1-\lambda}{\|v\|}_{Z_{\infty,\infty}^{-1+\epsilon}(T)}^{1-\lambda}{\|v\|}_{Z_{p,q}^{s}(T)}^{\lambda}\right),
        \end{aligned}
        \end{equation*}
        where we used \eqref{exi-bilinear-G-estimate} for the first inequality, and we used the interpolation \eqref{interpolation-holder} for the third and fourth inequalities; for the second inequality, we used the identity $u\otimes v=\mathsf{Bony}(u,v)$, together with the Bony estimates
        \begin{equation*}
            {\|\dot{T}_{u}v\|}_{\widetilde{L}_{T}^{\frac{2}{2-\epsilon}}\dot{B}_{p,q}^{s+1}(\mathbb{R}^{n})} \lesssim_{\varphi,s,\epsilon} {\|u\|}_{\widetilde{L}_{T}^{\infty}\dot{B}_{\infty,\infty}^{-1+\epsilon}(\mathbb{R}^{n})}{\|v\|}_{\widetilde{L}_{T}^{\frac{2}{2-\epsilon}}\dot{B}_{p,q}^{s+2-\epsilon}(\mathbb{R}^{n})},
        \end{equation*}
        \begin{equation*}
            {\|\dot{R}(u,v)\|}_{\widetilde{L}_{T}^{\frac{2}{2-\epsilon}}\dot{B}_{p,q}^{s+1}(\mathbb{R}^{n})} \lesssim_{\varphi,s} {\|u\|}_{\widetilde{L}_{T}^{\infty}\dot{B}_{\infty,\infty}^{-1+\epsilon}(\mathbb{R}^{n})}{\|v\|}_{\widetilde{L}_{T}^{\frac{2}{2-\epsilon}}\dot{B}_{p,q}^{s+2-\epsilon}(\mathbb{R}^{n})},
        \end{equation*}
        \begin{equation*}
            {\|\dot{T}_{v}u\|}_{\widetilde{L}_{T}^{\frac{2}{2-\epsilon}}\dot{B}_{p,q}^{s+1}(\mathbb{R}^{n})} \lesssim_{\varphi,s,\epsilon,\lambda} {\|v\|}_{\widetilde{L}_{T}^{\infty}\dot{B}_{\frac{p}{\lambda},\frac{q}{\lambda}}^{\lambda s+(1-\lambda)(-1+\epsilon)}(\mathbb{R}^{n})}{\|u\|}_{\widetilde{L}_{T}^{\frac{2}{2-\epsilon}}\dot{B}_{\frac{p}{1-\lambda},\frac{q}{1-\lambda}}^{(1-\lambda)(s+2-\epsilon)+\lambda}(\mathbb{R}^{n})}.
        \end{equation*}
        The estimate for ${\|B[v,u]\|}_{Z_{p,q}^{s}(T)}$ in \eqref{persistence-bilinear-estimate} is derived similarly, so the derivation of \eqref{persistence-bilinear-estimate} is complete.
    \end{enumerate}
\end{proof}

\section{Uniqueness}\label{besov-uniqueness-section}
The bilinear estimate \eqref{besov-uniqueness-bilinear-estimate} enables us to prove the following.
\begin{lemma}\label{besov-uniqueness-lemma-1}
    Assume that $(\alpha,\ell,\epsilon)\in{[1,\infty]}^{2}\times(0,1]$ satisfy \eqref{ale-conditions} with $\alpha<\infty$. Let $T\in(0,\infty)$, and let $C=C_{\varphi,\alpha,\ell,\epsilon}$ be the implied constant from \eqref{besov-uniqueness-bilinear-estimate}. If $u^{0}:(0,T)\rightarrow\mathcal{S}'(\mathbb{R}^{n})$ is a weakly* measurable function satisfying $4CT^{\frac{\epsilon}{2}}{\|u^{0}\|}_{X_{\ell,\epsilon}^{\alpha}(T)}<1$, then any weakly* measurable function $u:(0,T)\rightarrow\mathcal{S}'(\mathbb{R}^{n})$, satisfying ${\|u\|}_{X_{\ell,\epsilon}^{\alpha}(T)}<\infty$ and $u(t)=u^{0}(t)-B_{\alpha,\ell,\epsilon}[u,u](t)$ in $\mathcal{S}'(\mathbb{R}^{n})$ for all $t\in(0,T)$, must satisfy $CT^{\frac{\epsilon}{2}}{\|u\|}_{X_{\ell,\epsilon}^{\alpha}(T)}\leq\Lambda$, where $\Lambda\in[0,\frac{1}{2})$ is the smaller root of the quadratic
    \begin{equation}\label{uniqueness-Lambda}
        \Lambda = CT^{\frac{\epsilon}{2}}{\|u^{0}\|}_{X_{\ell,\epsilon}^{\alpha}(T)} + \Lambda^{2}.
    \end{equation}
\end{lemma}
\begin{proof}
    Define $g_{0}:=CT^{\frac{\epsilon}{2}}{\|u^{0}\|}_{X_{\ell,\epsilon}^{\alpha}(T)}$. For $t\in(0,T]$, define $g(t):=Ct^{\frac{\epsilon}{2}}{\|u\|}_{X_{\ell,\epsilon}^{\alpha}(t)}$ and $h(t):=Ct^{\frac{\epsilon}{2}}{\|B_{\alpha,\ell,\epsilon}[u,u]\|}_{\widetilde{L}_{t}^{\alpha}\dot{B}_{\ell,1}^{s_{\ell}+\epsilon+\frac{2}{\alpha}}(\mathbb{R}^{n})}$. Then $\lim_{t\searrow0}g(t)=0$, $h$ is continuous by dominated convergence, and the bilinear estimate \eqref{besov-uniqueness-bilinear-estimate} tells us that $Ct^{\frac{\epsilon}{2}}{\|B_{\alpha,\ell,\epsilon}[u,u]\|}_{X_{\ell,\epsilon}^{\alpha}(t)}\leq h(t)\leq{g(t)}^{2}$. Let $\Gamma$ be the larger root of the quadratic $\Gamma=g_{0}+\Gamma^{2}$. Since $u=u^{0}-B_{\alpha,\ell,\epsilon}[u,u]$, we have $g(t)\leq g_{0}+h(t)\leq g_{0}+{g(t)}^{2}$, so $g$ takes values in $[0,\Lambda]\cup[\Gamma,\infty)$. If $g(t)\leq\Lambda$ then $h(t)\leq{g(t)}^{2}\leq\Lambda^{2}$, while if $g(t)\geq\Gamma$ then $h(t)\geq g(t)-g_{0}\geq\Gamma-g_{0}=\Gamma^{2}$, so $h$ takes values in $[0,\Lambda^{2}]\cup[\Gamma^{2},\infty)$. But $h$ is continuous on $(0,T]$ with $\lim_{t\searrow0}h(t)=0$, so $h(t)\leq\Lambda^{2}$ for all $t\in(0,T]$. We conclude that $g(t)\leq\Lambda$ for all $t\in(0,T]$.
\end{proof}
\begin{lemma}\label{besov-uniqueness-lemma-2}
    Assume that $(\alpha,\ell,\epsilon)\in{[1,\infty]}^{2}\times(0,1]$ satisfy \eqref{ale-conditions} with $\alpha<\infty$. Let $T\in(0,\infty)$, and let $C=C_{\varphi,\alpha,\ell,\epsilon}$ be the implied constant from \eqref{besov-uniqueness-bilinear-estimate}. If $u^{0}:(0,T)\rightarrow\mathcal{S}'(\mathbb{R}^{n})$ is a weakly* measurable function satisfying $4CT^{\frac{\epsilon}{2}}{\|u^{0}\|}_{X_{\ell,\epsilon}^{\alpha}(T)}<1$, then there exists at most one weakly* measurable function $u:(0,T)\rightarrow\mathcal{S}'(\mathbb{R}^{n})$ satisfying ${\|u\|}_{X_{\ell,\epsilon}^{\alpha}(T)}<\infty$ and $u(t)=u^{0}(t)-B_{\alpha,\ell,\epsilon}[u,u](t)$ in $\mathcal{S}'(\mathbb{R}^{n})$ for all $t\in(0,T)$.
\end{lemma}
\begin{proof}
    Assume that $u,v$ are two such solutions. Then $u-v=-B_{\alpha,\ell,\epsilon}[u,u-v]-B_{\alpha,\ell,\epsilon}[u-v,v]$, so by \eqref{besov-uniqueness-bilinear-estimate} and Lemma \ref{besov-uniqueness-lemma-1} we have ${\|u-v\|}_{X_{\ell,\epsilon}^{\alpha}(T)}\leq2\Lambda{\|u-v\|}_{X_{\ell,\epsilon}^{\alpha}(T)}$, where $\Lambda\in[0,\frac{1}{2})$ is the smaller root of the quadratic \eqref{uniqueness-Lambda}. Therefore ${\|u-v\|}_{X_{\ell,\epsilon}^{\alpha}(T)}=0$, so $u(t)=v(t)$ in $\mathcal{S}'(\mathbb{R}^{n})$ for almost every $t\in(0,T)$. By \eqref{besov-uniqueness-bilinear-estimate}, it follows that $B_{\alpha,\ell,\epsilon}[u,u](t)=B_{\alpha,\ell,\epsilon}[v,v](t)$ in $\mathcal{S}'(\mathbb{R}^{n})$ for all $t\in(0,T)$. Since we assumed that $u(t)+B_{\alpha,\ell,\epsilon}[u,u](t)=v(t)+B_{\alpha,\ell,\epsilon}[v,v](t)$ in $\mathcal{S}'(\mathbb{R}^{n})$ for all $t\in(0,T)$, we conclude that $u(t)=v(t)$ in $\mathcal{S}'(\mathbb{R}^{n})$ for all $t\in(0,T)$.
\end{proof}
We can now give the
\begin{proof}[Proof of Proposition \ref{besov-uniqueness-theorem}]
    We may suppose without loss of generality that $\alpha<\infty$ (a small decrease in $\alpha$ can be counteracted by a small decrease in $\epsilon$, such that $\epsilon+\frac{2}{\alpha}$ remains the same). Assume that $u,v$ are two such solutions. By \eqref{besov-uniqueness-bilinear-estimate}, we have ${\|S[f]\|}_{X_{\ell,\epsilon}^{\alpha}(T')}<\infty$ for all $T'\in(0,T)$.

    Let $C=C_{\varphi,\alpha,\ell,\epsilon}$ be the implied constant from \eqref{besov-uniqueness-bilinear-estimate}. By Lemma \ref{besov-uniqueness-lemma-2}, if $4C{(T')}^{\frac{\epsilon}{2}}{\|S[f]\|}_{X_{\ell,\epsilon}^{\alpha}(T')}<1$ then we have uniqueness on the time interval $(0,T')$. This inequality holds for sufficiently small $T'$, so the quantity $T_{0}:=\inf\left\{T'\in(0,T)\text{ : }u(t)\neq v(t)\right\}$ is strictly positive.

    Suppose for the sake of contradiction that $T_{0}<T$. Then we can choose $\eta\in(0,T_{0})$ sufficiently small such that
    \begin{equation}\label{besov-uniqueness-blowup}
        4C{(2\eta)}^{\frac{\epsilon}{2}}\left({\|\tau_{T_{0}-\eta}u\|}_{X_{\ell,\epsilon}^{\alpha}(2\eta)}+C{(2\eta)}^{\frac{\epsilon}{2}}{\|\tau_{T_{0}-\eta}u\|}_{X_{\ell,\epsilon}^{\alpha}(2\eta)}^{2}\right) < 1,
    \end{equation}
    where we use the notation $\tau_{t_{0}}u(t)=u(t+t_{0})$. By the semigroup properties for the operators $S$ and $B$, we have $S[u(T_{0}-\eta)]=\tau_{T_{0}-\eta}u+B_{\alpha,\ell,\epsilon}[\tau_{T_{0}-\eta}u,\tau_{T_{0}-\eta}u]=\tau_{T_{0}-\eta}v+B_{\alpha,\ell,\epsilon}[\tau_{T_{0}-\eta}v,\tau_{T_{0}-\eta}v]$. By \eqref{besov-uniqueness-bilinear-estimate} and \eqref{besov-uniqueness-blowup}, we then have $4C{(2\eta)}^{\frac{\epsilon}{2}}{\|S[u(T_{0}-\eta)]\|}_{X_{\ell,\epsilon}^{\alpha}(2\eta)}<1$. By Lemma \ref{besov-uniqueness-lemma-2}, it follows that $u=v$ on the time interval $(T_{0}-\eta,T_{0}+\eta)$, so $u=v$ on the time interval $(0,T_{0}+\eta)$, which contradicts the definition of $T_{0}$.
\end{proof}
\section{Local existence and blowup rates}\label{besov-existence-section}
The bilinear estimates \eqref{besov-existence-bilinear-estimate}-\eqref{persistence-bilinear-estimate} and the embedding ${\|u\|}_{Y_{\epsilon/2}(T)}\lesssim_{\epsilon}T^{\epsilon/4}{\|u\|}_{Z_{\infty,\infty}^{-1+\epsilon}(T)}$ enable us to give the
\begin{proof}[Proof of Theorem \ref{besov-local-existence-theorem}]
    We start by proving part (i). Let $A_{\varphi}$ be the implied constant from \eqref{besov-existence-bilinear-estimate}, and let $\widetilde{A}_{\epsilon}$ be the implied constant from the embedding ${\|u\|}_{Y_{\epsilon/2}(T)}\lesssim_{\epsilon}T^{\epsilon/4}{\|u\|}_{Z_{\infty,\infty}^{-1+\epsilon}(T)}$. Defining the norms
    \begin{equation*}
        {\|u\|}_{V_{\epsilon}^{0}(T)} := \frac{A_{\varphi}}{\epsilon(1-\epsilon)}T^{\frac{\epsilon}{2}}{\|u\|}_{Z_{\infty,\infty}^{-1+\epsilon}(T)}, \quad {\|u\|}_{V_{\epsilon}^{1}(T)} := \frac{A_{\varphi}\widetilde{A}_{\epsilon}}{\epsilon}T^{\frac{\epsilon}{4}}{\|u\|}_{Y_{\epsilon/2}(T)},
    \end{equation*}
    the estimate \eqref{besov-existence-bilinear-estimate} and the embedding ${\|u\|}_{Y_{\epsilon/2}(T)}\leq\widetilde{A}_{\epsilon}T^{\epsilon/4}{\|u\|}_{Z_{\infty,\infty}^{-1+\epsilon}(T)}$ tell us that
    \begin{equation}\label{exi-proof-bilinear-estimate}
        {\|B[u,v]\|}_{V_{\epsilon}^{j}(T)} \leq {\|u\|}_{V_{\epsilon}^{j}(T)}{\|v\|}_{V_{\epsilon}^{j}(T)}, \quad {\|B[u,v]\|}_{Z_{\infty,\infty}^{-1+\epsilon}(T)} \lesssim_{\varphi,\epsilon} T^{-\epsilon/2}{\|u\|}_{V_{\epsilon}^{j}(T)}{\|v\|}_{V_{\epsilon}^{j}(T)},
    \end{equation}
    whenever $j\in\{0,1\}$ and $u,v\in Z_{\infty,\infty}^{-1+\epsilon}(T)$.

    By the heat estimates \eqref{besov-heat-estimate-1} and \eqref{besov-heat-estimate-3}, and the embedding ${\|f\|}_{\dot{B}_{\infty,\infty}^{0}(\mathbb{R}^{n})}\lesssim_{\varphi}{\|f\|}_{L^{\infty}(\mathbb{R}^{n})}$, for all $f\in\dot{B}_{\infty,\infty}^{-1+\epsilon}(\mathbb{R}^{n})$ we have
    \begin{equation*}
        T^{\frac{\epsilon}{2}}{\|S[f]\|}_{Z_{\infty,\infty}^{-1+\epsilon}(T)} \lesssim_{n} T^{\frac{\epsilon}{2}}{\|f\|}_{\dot{B}_{\infty,\infty}^{-1+\epsilon}(\mathbb{R}^{n})},
    \end{equation*}
    \begin{equation*}
        T^{\frac{\epsilon}{4}}{\|S[f]\|}_{Y_{\epsilon/2}(T)} \leq T^{\frac{1}{2}}{\|S[f]\|}_{\widetilde{L}_{T}^{\frac{2}{\epsilon}}\dot{B}_{\infty,\infty}^{\epsilon}(\mathbb{R}^{n})\cap L_{T}^{\infty}L^{\infty}(\mathbb{R}^{n})} \lesssim_{\varphi} T^{\frac{1}{2}}{\|f\|}_{L^{\infty}(\mathbb{R}^{n})},
    \end{equation*}
    so we can choose constants $C_{\varphi}$ and $\widetilde{C}_{\varphi,\epsilon}$ for which we have the implications
    \begin{equation*}
        \frac{C_{\varphi}}{\epsilon(1-\epsilon)}T^{\frac{\epsilon}{2}}{\|f\|}_{\dot{B}_{\infty,\infty}^{-1+\epsilon}(\mathbb{R}^{n})} < 1 \quad \Rightarrow \quad \frac{4A_{\varphi}}{\epsilon(1-\epsilon)}T^{\frac{\epsilon}{2}}{\|S[f]\|}_{Z_{\infty,\infty}^{-1+\epsilon}(T)} < 1,
    \end{equation*}
    \begin{equation*}
        \widetilde{C}_{\varphi,\epsilon}T^{\frac{1}{2}}{\|f\|}_{L^{\infty}(\mathbb{R}^{n})} < 1 \quad \Rightarrow \quad \frac{4A_{\varphi}\widetilde{A}_{\epsilon}}{\epsilon}T^{\frac{\epsilon}{4}}{\|S[f]\|}_{Y_{\epsilon/2}(T)} < 1.
    \end{equation*}
    We therefore have the following: if $f\in\dot{B}_{\infty,\infty}^{-1+\epsilon}(\mathbb{R}^{n})$ satisfies \eqref{besov-local-existence-condition}, then $S[f]\in Z_{\infty,\infty}^{-1+\epsilon}(T)$ satisfies $\min_{j\in\{0,1\}}{\|S[f]\|}_{V_{\epsilon}^{j}(T)}<1/4$.

    Using the fact that $S[f]\in Z_{\infty,\infty}^{-1+\epsilon}(T)$ satisfies $\min_{j\in\{0,1\}}{\|S[f]\|}_{V_{\epsilon}^{j}(T)}<1/4$, we will construct a solution using Picard iteration. Fix a value $j\in\{0,1\}$ such that $4{\|S[f]\|}_{V_{\epsilon}^{j}(T)}<1$, and consider the Picard scheme given by $u^{0}=S[f]$, and $u^{m+1}=S[f]-B[u^{m},u^{m}]$ for $m\geq0$. By the bilinear estimate \eqref{exi-proof-bilinear-estimate}, and by induction on $m$, we have that $u^{m}\in Z_{\infty,\infty}^{-1+\epsilon}(T)$ satisfies ${\|u^{m}\|}_{V_{\epsilon}^{j}(T)}\leq\Lambda$, where $\Lambda\in[0,\frac{1}{2})$ is the smaller root of the quadratic $\Lambda={\|S[f]\|}_{V_{\epsilon}^{j}(T)}+\Lambda^{2}$. By writing
    \begin{equation*}
        u^{m+1}-u^{m} = -B[u^{m},u^{m}-u^{m-1}] - B[u^{m}-u^{m-1},u^{m-1}],
    \end{equation*}
    the bilinear estimate \eqref{exi-proof-bilinear-estimate} then tells us that we have ${\|u^{m+1}-u^{m}\|}_{V_{\epsilon}^{j}(T)}\leq2\Lambda{\|u^{m}-u^{m-1}\|}_{V_{\epsilon}^{j}(T)}$ and ${\|u^{m+1}-u^{m}\|}_{Z_{\infty,\infty}^{-1+\epsilon}(T)} \lesssim_{\varphi,\epsilon} T^{-\epsilon/2}\Lambda{\|u^{m}-u^{m-1}\|}_{V_{\epsilon}^{j}(T)}$ for $m\geq1$, from which we deduce that $\sum_{m=1}^{\infty}{\|u^{m+1}-u^{m}\|}_{Z_{\infty,\infty}^{-1+\epsilon}(T)}\lesssim_{\varphi,\epsilon}T^{-\epsilon/2}\sum_{m=0}^{\infty}{\|u^{m+1}-u^{m}\|}_{V_{\epsilon}^{j}(T)}<\infty$. By Lemma \ref{Banach-space}, it follows that $u^{m}$ converges in $Z_{\infty,\infty}^{-1+\epsilon}(T)$ to some $u$ satisfying $u(t)=S[f](t)-B[u,u](t)$ in $\mathcal{S}'(\mathbb{R}^{n})$ for all $t\in(0,T)$.

    To complete the proof of part (i), we will adapt the ideas of \cite[Theorem 9.11]{lemarie2016} to show that if $f\in\dot{B}_{p,q}^{s}(\mathbb{R}^{n})\cap\dot{B}_{\infty,\infty}^{-1+\epsilon}(\mathbb{R}^{n})$ satisfies \eqref{besov-local-existence-condition}, then the Picard scheme described above satisfies $u^{0}\in Z_{p,q}^{s}(T)$ and $\sum_{m=0}^{\infty}{\|u^{m+1}-u^{m}\|}_{Z_{p,q}^{s}(T)}<\infty$, which implies (by Lemma \ref{Banach-space}) that the limit $u=\lim_{m\rightarrow\infty}u^{m}$ belongs to $Z_{p,q}^{s}(T)$. By \eqref{besov-heat-estimate-1} and \eqref{persistence-bilinear-estimate}, we have $u^{m}\in Z_{p,q}^{s}(T)$ for all $m\geq0$. Let $u^{-1}=0$, and define (for $m\geq0$) the quantities $\alpha_{m}={\|u^{m}-u^{m-1}\|}_{Z_{p,q}^{s}(T)}$, $N_{m}=\sum_{k=0}^{m}\alpha_{k}$, $\beta_{m}=\sup_{k\leq m}T^{\epsilon/2}{\|u^{k}\|}_{Z_{\infty,\infty}^{-1+\epsilon}(T)}$, and $\gamma_{m}=T^{\epsilon/2}{\|u^{m}-u^{m-1}\|}_{Z_{\infty,\infty}^{-1+\epsilon}(T)}$. Choose $\lambda\in(0,1)$ such that $\lambda(s+1-\epsilon)<1-\epsilon$. For all $m\geq0$ we have
    \begin{equation*}
        u^{m+1}-u^{m} = -B[u^{m},u^{m}-u^{m-1}] - B[u^{m}-u^{m-1},u^{m-1}],
    \end{equation*}
    so by \eqref{persistence-bilinear-estimate}, and the estimate $\sup_{k\leq m}{\|u^{k}\|}_{Z_{p,q}^{s}(T)}\leq N_{m}$, we have
    \begin{equation*}
        \alpha_{m+1} \lesssim_{\varphi,s,\epsilon,\lambda} \gamma_{m}N_{m}+\gamma_{m}^{\lambda}\alpha_{m}^{1-\lambda}\beta_{m}^{1-\lambda}N_{m}^{\lambda},
    \end{equation*}
    so by Young's product inequality we have
    \begin{equation*}
        \alpha_{m+1} \leq \frac{1}{2}\alpha_{m} + K\gamma_{m}N_{m}\left(1+\beta_{\infty}^{\frac{1}{\lambda}-1}\right)
    \end{equation*}
    for some constant $K=K_{\varphi,s,\epsilon,\lambda}>0$. For $m\geq1$ we therefore have
    \begin{equation*}
    \begin{aligned}
        N_{m+1}-\frac{1}{2}N_{m} &= \alpha_{m+1}-\frac{1}{2}\alpha_{m}+N_{m}-\frac{1}{2}N_{m-1} \\
        &\leq K\gamma_{m}N_{m}\left(1+\beta_{\infty}^{\frac{1}{\lambda}-1}\right) + N_{m}-\frac{1}{2}N_{m-1} \\
        &= \left(1+K\gamma_{m}\left(1+\beta_{\infty}^{\frac{1}{\lambda}-1}\right)\right)\left(N_{m}-\frac{1}{2}N_{m-1}\right)+\frac{1}{2}K\gamma_{m}N_{m-1}\left(1+\beta_{\infty}^{\frac{1}{\lambda}-1}\right) \\
        &\leq \left(1+2K\gamma_{m}\left(1+\beta_{\infty}^{\frac{1}{\lambda}-1}\right)\right)\left(N_{m}-\frac{1}{2}N_{m-1}\right),
    \end{aligned}
    \end{equation*}
    so
    \begin{equation*}
    \begin{aligned}
        \frac{1}{2}N_{m} &\leq N_{m+1} - \frac{1}{2}N_{m} \\
        &\leq \left(N_{1}-\frac{1}{2}N_{0}\right)\prod_{k=1}^{m}\left(1+2K\gamma_{m}\left(1+\beta_{\infty}^{\frac{1}{\lambda}-1}\right)\right) \\
        &\leq \left(N_{1}-\frac{1}{2}N_{0}\right)\exp\left(\sum_{k=1}^{m}2K\gamma_{m}\left(1+\beta_{\infty}^{\frac{1}{\lambda}-1}\right)\right),
    \end{aligned}
    \end{equation*}
    so
    \begin{equation*}
        \frac{1}{2}N_{\infty} \leq \left(N_{1}-\frac{1}{2}N_{0}\right)\exp\left(\sum_{k=1}^{\infty}2K\gamma_{m}\left(1+\beta_{\infty}^{\frac{1}{\lambda}-1}\right)\right) < \infty,
    \end{equation*}
    which completes the proof of part (i).

    We now turn our attention to part (ii). Let us use the notation $\rho_{\varphi,\epsilon}(T,f)$ to denote the left hand side of \eqref{besov-local-existence-condition}. Suppose for the sake of contradiction that $T_{f}^{*}<\infty$, and that there exists $t_{0}\in(0,T_{f}^{*})$ satisfying $\rho_{\varphi,\epsilon}(T_{f}^{*}-t_{0},u(t_{0}))<1$. By continuity, we can choose $T\in(T_{f}^{*},\infty)$ such that $\rho_{\varphi,\epsilon}(T-t_{0},u(t_{0}))<1$. Then there exists $w\in Z_{p,q}^{s}(T-t_{0})\cap Z_{\infty,\infty}^{-1+\epsilon}(T-t_{0})$ satisfying $w=S[u(t_{0})]-B[w,w]$. By the semigroup properties for $S$ and $B$ we have $\tau_{t_{0}}u=S[u(t_{0})]-B[\tau_{t_{0}}u,\tau_{t_{0}}u]$. By uniqueness, therefore $w(t)=u(t+t_{0})$ for all $t\in(t_{0},T_{f}^{*})$. Again by the semigroup properties for $S$ and $B$, we see that
    \begin{equation*}
        v(t) = \left\{\begin{array}{ll} u(t) & \text{if }t\in(0,T_{f}^{*}), \\ w(t-t_{0}) & \text{if }t\in(t_{0},T), \end{array}\right.
    \end{equation*}
    defines $v\in Z_{p,q}^{s}(T)\cap Z_{\infty,\infty}^{-1+\epsilon}(T)$ satisfying $v=S[f]-B[v,v]$, which contradicts the definiiton of $T_{f}^{*}$, and completes the proof of part (ii). Therefore the proof of Theorem \ref{besov-local-existence-theorem} is complete.
\end{proof}
We now give the
\begin{proof}[Proof of Corollary \ref{further-uni}]
    By Remark \ref{further-uni-remark}(i), we don't need to worry about the distinction between the bilinear operators $B$ and $\widetilde{B}$. Let $u\in\cap_{T'\in(0,T_{f}^{*})}Z_{\ell,\infty}^{s_{\ell}+\eta}(T')$ be the solution with maximal existence time $T^*_f>0$ from Theorem \ref{besov-local-existence-theorem}, and note that for $T'\in(0,\infty)$ and $\beta\in[1,\infty)$ we have
    \begin{equation}\label{Z-embedding}
    \begin{aligned}
        Z_{\ell,\infty}^{s_{\ell}+\eta}(T') &\subseteq \widetilde{L}_{T'}^{\beta}\dot{B}_{\ell,\infty}^{s_{\ell}+\eta+\frac{2}{\beta}}(\mathbb{R}^{n})\cap\widetilde{L}_{T'}^{\infty}\dot{B}_{\ell,\infty}^{s_{\ell}+\eta}(\mathbb{R}^{n}) \\
        &\subseteq \widetilde{L}_{T'}^{\beta}\dot{B}_{\ell,\infty}^{s_{\ell}+\eta+\frac{2}{\beta}}(\mathbb{R}^{n})\cap\widetilde{L}_{T'}^{\beta}\dot{B}_{\ell,\infty}^{s_{\ell}+\eta}(\mathbb{R}^{n}) \\
        &\subseteq \left(\cap_{\zeta\in[\eta-\frac{2}{\beta},\eta]}\widetilde{L}_{T'}^{\beta}\dot{B}_{\ell,\infty}^{s_{\ell}+\zeta+\frac{2}{\beta}}(\mathbb{R}^{n})\right)\cap\left(\cap_{\zeta\in(\eta-\frac{2}{\beta},\eta)}\widetilde{L}_{T'}^{\beta}\dot{B}_{\ell,1}^{s_{\ell}+\zeta+\frac{2}{\beta}}(\mathbb{R}^{n})\right).
    \end{aligned}
    \end{equation}
    We then prove (i) and (ii) as follows.
    \begin{enumerate}[label=(\roman*)]
        \item
        If $-\left(s_{\ell}+\frac{2}{\alpha}\right)<\epsilon\leq\eta$, then the embedding \eqref{Z-embedding} (with $\beta=\alpha$ and $\zeta=\epsilon$) tells us that\footnote{The inequality $\epsilon\leq\eta\leq\epsilon+\frac{2}{\alpha}$ ensures that we can apply \eqref{Z-embedding} to obtain $Z_{\ell,\infty}^{s_{\ell}+\eta}(T')\subseteq\widetilde{L}_{T'}^{\alpha}\dot{B}_{\ell,\infty}^{s_{\ell}+\epsilon+\frac{2}{\alpha}}(\mathbb{R}^{n})$. The inequality $s_{\ell}+\epsilon+\frac{2}{\alpha}>0$ ensures that $\widetilde{L}_{T'}^{\alpha}\dot{B}_{\ell,\infty}^{s_{\ell}+\epsilon+\frac{2}{\alpha}}(\mathbb{R}^{n})=X_{\ell,\epsilon}^{\alpha}(T')$.} $Z_{\ell,\infty}^{s_{\ell}+\eta}(T')\subseteq\widetilde{L}_{T'}^{\alpha}\dot{B}_{\ell,\infty}^{s_{\ell}+\epsilon+\frac{2}{\alpha}}(\mathbb{R}^{n})=X_{\ell,\epsilon}^{\alpha}(T')$. If $\epsilon<\eta<\epsilon+\frac{2}{\alpha}$, then the embedding \eqref{Z-embedding} (with $\beta=\alpha$ and $\zeta=\epsilon$) tells us that $Z_{\ell,\infty}^{s_{\ell}+\eta}(T')\subseteq\widetilde{L}_{T'}^{\alpha}\dot{B}_{\ell,1}^{s_{\ell}+\epsilon+\frac{2}{\alpha}}(\mathbb{R}^{n})\subseteq X_{\ell,\epsilon}^{\alpha}(T')$. In either case, it follows that the solution $u$ belongs to $\cap_{T'\in(0,T_{f}^{*})}X_{\ell,\epsilon}^{\alpha}(T')$, so that $\widetilde{T}\geq T^*_f >0$.
        \item
        Choose $\zeta\in(0,\eta)\cap(0,\epsilon]$, and define $\beta:=2{\left(\epsilon+\frac{2}{\alpha}-\zeta\right)}^{-1}$, so that $\epsilon+\frac{2}{\alpha}=\zeta+\frac{2}{\beta}$ and $X_{\ell,\epsilon}^{\alpha}(T')\subseteq X_{\ell,\zeta}^{\beta}(T')$ for all $T'\in(0,\infty)$. If $\eta=\epsilon+\frac{2}{\alpha}$, then the assumption $\eta<s_{\ell}+2\left(\epsilon+\frac{2}{\alpha}\right)$ implies that $s_{\ell}+\zeta+\frac{2}{\beta}>0$, so we have the embedding $Z_{\ell,\infty}^{s_{\ell}+\eta}(T')\subseteq\widetilde{L}_{T'}^{\infty}\dot{B}_{\ell,\infty}^{s_{\ell}+\zeta+\frac{2}{\beta}}(\mathbb{R}^{n})\subseteq\widetilde{L}_{T'}^{\beta}\dot{B}_{\ell,\infty}^{s_{\ell}+\zeta+\frac{2}{\beta}}(\mathbb{R}^{n})=X_{\ell,\zeta}^{\beta}(T')$. If $\eta<\epsilon+\frac{2}{\alpha}$, then the embedding \eqref{Z-embedding} tells us that $Z_{\ell,\infty}^{s_{\ell}+\eta}(T')\subseteq\widetilde{L}_{T'}^{\beta}\dot{B}_{\ell,1}^{s_{\ell}+\zeta+\frac{2}{\beta}}(\mathbb{R}^{n})\subseteq X_{\ell,\zeta}^{\beta}(T')$. In either case, it follows that the solution $u$ belongs to $\cap_{T'\in(0,T_{f}^{*})}X_{\ell,\zeta}^{\beta}(T')$, and that $u(t_{0})\in\dot{B}_{\ell,q}^{s_{\ell}+\epsilon+\frac{2}{\alpha}}(\mathbb{R}^{n})$ for almost every $t_{0}\in(0,T_{f}^{*})$, where we define
        \begin{equation*}
            q := \left\{\begin{array}{ll}\infty & \text{if }\eta=\epsilon+\frac{2}{\alpha}, \\ 1 & \text{if }\eta<\epsilon+\frac{2}{\alpha}. \end{array}\right.
        \end{equation*}
        For every $t_{1}\in(0,T_{f}^{*})$, we can choose $t_{0}\in(0,t_{1})$ such that $u(t_{0})\in\dot{B}_{\ell,q}^{s_{\ell}+\epsilon+\frac{2}{\alpha}}(\mathbb{R}^{n})$. Now $\tau_{t_{0}}u\in\cap_{T'\in(0,T_{f}^{*}-t_{0})}Z_{\infty,\infty}^{-1+\eta}(T')$ satisfies the equation $\tau_{t_{0}}u=S[u(t_{0})]-B[\tau_{t_{0}}u,\tau_{t_{0}}u]$ with (recall \eqref{chemin-lerner-minkowski}) $u(t_{0})\in\dot{B}_{\ell,q}^{s_{\ell}+\epsilon+\frac{2}{\alpha}}(\mathbb{R}^{n})\cap\dot{B}_{\infty,\infty}^{-1+\eta}(\mathbb{R}^{n})$, so by propagation of regularity (Remark \ref{besov-existence-remark}(ii)) we have that $\tau_{t_{0}}u\in\cap_{T'\in(0,T_{f}^{*}-t_{0})}Z_{\ell,q}^{s_{\ell}+\epsilon+\frac{2}{\alpha}}(T')$. It follows that $\tau_{t_{1}}u\in\cap_{T'\in(0,T_{f}^{*}-t_{1})}\widetilde{L}_{T'}^{\infty}\dot{B}_{\ell,q}^{s_{\ell}+\epsilon+\frac{2}{\alpha}}(\mathbb{R}^{n})\subseteq\cap_{T'\in(0,T_{f}^{*}-t_{1})}X_{\ell,\epsilon}^{\alpha}(T')$ for every $t_{1}\in(0,T_{f}^{*})$.

        Now suppose that $\widetilde{T}>0$, and let $v$ be the solution belonging to the space $\cap_{T'\in(0,\widetilde{T})}X_{\ell,\epsilon}^{\alpha}(T')\subseteq\cap_{T'\in(0,\widetilde{T})}X_{\ell,\zeta}^{\beta}(T')$. By \eqref{besov-heat-estimate-1} we have $S[f]\in L_{\infty}^{\infty}\dot{B}_{\ell,\infty}^{s_{\ell}+\eta}(\mathbb{R}^{n})$, while by \eqref{besov-uniqueness-bilinear-estimate}
        \begin{equation*}
            B[v,v] \in L_{T'}^{\beta}\dot{B}_{\ell,1}^{s_{\ell}+\zeta+\frac{2}{\beta}}(\mathbb{R}^{n})\cap L_{T'}^{\infty}\dot{B}_{\ell,1}^{s_{\ell}+\zeta}(\mathbb{R}^{n}) \subseteq L_{T'}^{\frac{2}{\eta-\zeta}}\dot{B}_{\ell,1}^{s_{\ell}+\eta}(\mathbb{R}^{n})
        \end{equation*}
        for all $T'\in(0,\widetilde{T})$. We deduce that $v\in\cap_{T'\in(0,\widetilde{T})}L_{T'}^{2/\eta}\dot{B}_{\ell,\infty}^{s_{\ell}+\eta}(\mathbb{R}^{n})$.

        By uniqueness in $X_{\ell,\zeta}^{\beta}$, the solutions $u$ and $v$ coincide on their common interval of existence. Since $v\in\cap_{T'\in(0,\widetilde{T})}L_{T'}^{2/\eta}\dot{B}_{\ell,\infty}^{s_{\ell}+\eta}(\mathbb{R}^{n})$, the blowup estimate ${\|u(t_{0})\|}_{\dot{B}_{\infty,\infty}^{-1+\eta}(\mathbb{R}^{n})}\gtrsim_{\varphi,\epsilon}{(T_{f}^{*}-t_{0})}^{-\eta/2}$ rules out the possibility that $T_{f}^{*}<\widetilde{T}$. Since $v\in\cap_{T'\in(0,\widetilde{T})}X_{\ell,\epsilon}^{\alpha}(T')$ and $\tau_{t_{0}}u\in\cap_{T'\in(0,T_{f}^{*}-t_{0})}X_{\ell,\epsilon}^{\alpha}(T')$, we have $u\in\cap_{T'\in(0,T_{f}^{*})}X_{\ell,\epsilon}^{\alpha}(T')$, which rules out the possibility that $T_{f}^{*}>\widetilde{T}$. Therefore $T_{f}^{*}=\widetilde{T}$, and the solutions $u$ and $v$ coincide.
    \end{enumerate}
\end{proof}

\ \\\\\\

\appendix
\section{Appendix: Making sense of the fixed point problem}
\subsection{Weak* measurability}
We prove the following useful lemmas. These lemmas are proved in \cite[Lemmas 2.1-2.3]{davies-notions-of-solution} in the more general context of Lorentz spaces.
\begin{lemma}\label{measurability-duality}
    Let $(E,\mathcal{E})$ be a measurable space. For $p\in[1,\infty]$, a function $u:E\rightarrow L^{p}(\mathbb{R}^{n})$ is weakly* measurable if and only if the map $t\rightarrow\langle u(t),\phi\rangle$ is measurable for all $\phi\in L^{p'}(\mathbb{R}^{n})$.
\end{lemma}
\begin{proof}
    Let $u:E\rightarrow L^{p}(\mathbb{R}^{n})$ be a weakly* measurable function, and let $\phi\in L^{p'}(\mathbb{R}^{n})$. Define the cutoff function $\rho_{R}(x)=\rho(x/R)$ for $R\in\mathbb{Q}_{>0}$, where $\rho\in C_{c}^{\infty}(\mathbb{R}^{n})$ satisfies $\rho(x)=1$ for $|x|<1$ and $\rho(x)=0$ for $|x|>2$. Define also the approximate identity $\eta_{\epsilon}(x)=\epsilon^{-n}\eta(x/\epsilon)$ for $\epsilon\in\mathbb{Q}_{>0}$, where $\eta\in C_{c}^{\infty}(\mathbb{R}^{n})$ satisfies $\int_{\mathbb{R}^{n}}\eta(x)\,\mathrm{d}x=1$. Then the function
    \begin{equation*}
        \langle u(t),\phi\rangle = \lim_{\epsilon\rightarrow0}\langle u(t),\eta_{\epsilon}*\phi\rangle = \lim_{\epsilon\rightarrow0}\lim_{R\rightarrow\infty}\langle u(t),\rho_{R}\cdot(\eta_{\epsilon}*\phi)\rangle
    \end{equation*}
    is measurable, where the limit $R\rightarrow\infty$ follows from dominated convergence, and the limit $\epsilon\rightarrow0$ follows from approximation of identity in $L^{q}(\mathbb{R}^{n})$ for $q<\infty$. (In the case $p=1$, use Fubini's theorem to write $\langle u(t),\eta_{\epsilon}*\phi\rangle=\langle \eta_{\epsilon}*u(t),\phi\rangle$ before taking the limit $\epsilon\rightarrow0$.).
\end{proof}
\begin{lemma}\label{measurable-norms}
    Let $(E,\mathcal{E})$ be a measurable space. If $p\in[1,\infty]$, and $u:E\rightarrow L^{p}(\mathbb{R}^{n})$ is a weakly* measurable function, then the function of $t$ given by $t\mapsto{\|u(t)\|}_{L^{p}(\mathbb{R}^{n})}=\sup_{{\|\phi\|}_{L^{p'}(\mathbb{R}^{n})}\leq1}\langle u(t),\phi\rangle$ is measurable.
\end{lemma}
\begin{proof}
    In the case $p>1$, we choose a countable dense subset $A\subseteq\{\phi\in L^{p'}(\mathbb{R}^{n})\text{ : }{\|\phi\|}_{L^{p'}(\mathbb{R}^{n})}\leq1\}$, so the function ${\|u(t)\|}_{L^{p}(\mathbb{R}^{n})}=\sup_{\phi\in A}\langle u(t),\phi\rangle$ is measurable.

    In the case $p=1$, we will construct a countable family $A\subseteq B:=\{F\in L^{\infty}(\mathbb{R}^{n})\text{ : }{\|F\|}_{L^{\infty}(\mathbb{R}^{n})}\leq1\}$ such that for each $F\in B$ there exists a sequence ${(F_{m})}_{m\geq1}$ in $A$ satisfying $\langle F_{m},\phi\rangle\overset{m\rightarrow\infty}{\rightarrow}\langle F,\phi\rangle$ for all $\phi\in L^{1}(\mathbb{R}^{n})$. If such a family $A$ exists, then the function ${\|u(t)\|}_{L^{1}(\mathbb{R}^{n})}=\sup_{F\in A}\langle F,u(t)\rangle$. We construct $A$ as follows. Let ${(\phi_{m})}_{m\geq1}$ be dense in $L^{1}(\mathbb{R}^{n})$, and observe (by separability of $\mathbb{R}^{m}$) that each family $\left\{{(\langle F,\phi_{l}\rangle)}_{l\leq m}\text{ : }F\in B\right\}\subseteq\mathbb{R}^{m}$ has a dense subset $\left\{{(\langle F_{k,m},\phi_{l}\rangle)}_{l\leq m}\text{ : }k\geq1\right\}$. We claim that $A={(F_{k,m})}_{k,m\geq1}$ is as required. Given $F\in B$, for each $m\geq1$ there exists $k(m)$ such that $\left|\langle F_{k(m),m}-F,\phi_{l}\rangle\right|<\frac{1}{m}$ for all $l\leq m$, so for any $\phi\in L^{1}(\mathbb{R}^{n})$ and $\epsilon>0$ we can choose $l(\epsilon)$ with ${\|\phi_{l(\epsilon)}-\phi\|}_{L^{1}(\mathbb{R}^{n})}<\epsilon$ to obtain
    \begin{equation*}
        \limsup_{m\rightarrow\infty}\left|\langle F_{k(m),m}-F,\phi\rangle\right|\leq\limsup_{m\rightarrow\infty}\left(\left|\langle F_{k(m),m}-F,\phi_{l(\epsilon)}\rangle\right|+\left|\langle F_{k(m),m}-F,\phi_{l(\epsilon)}-\phi\rangle\right|\right) \leq 2\epsilon.
    \end{equation*}
\end{proof}
\begin{lemma}\label{pettis}
    (Pettis' theorem). Let $(E,\mathcal{E})$ be a measurable space. If $p\in[1,\infty)$, and $u:E\rightarrow L^{p}(\mathbb{R}^{n})$ is a weakly* measurable function, then there exist measurable functions $u_{m}:E\rightarrow L^{p}(\mathbb{R}^{n})$ with finite image, satisfying ${\|u_{m}(t)\|}_{L^{p}(\mathbb{R}^{n})}\leq{\|u(t)\|}_{L^{p}(\mathbb{R}^{n})}$ and ${\|u_{m}(t)-u(t)\|}_{L^{p}(\mathbb{R}^{n})}\overset{m\rightarrow\infty}{\rightarrow}0$ for all $t\in E$.
\end{lemma}
\begin{proof}
    Let $A$ be a countable dense subset of $L^{p}(\mathbb{R}^{n})$, let $\mathbb{Q}A=\left\{qa\text{ : }q\in\mathbb{Q},\,a\in A\right\}$, and let ${(f_{j})}_{j\geq0}$ be an enumeration of $\mathbb{Q}A$ with $f_{0}=0$. For each $m\geq0$ and $t\in E$, define the non-empty set
    \begin{equation*}
        K(m,t) := \left\{j\in\{0,\cdots,m\}\text{ : }{\|f_{j}\|}_{L^{p}(\mathbb{R}^{n})}\leq{\|u(t)\|}_{L^{p}(\mathbb{R}^{n})}\right\}
    \end{equation*}
    and the integer
    \begin{equation*}
        k(m,t) := \min\left\{j\in K(m,t)\text{ : }{\|f_{j}-u(t)\|}_{L^{p}(\mathbb{R}^{n})}=\min_{i\in K(m,t)}{\|f_{i}-u(t)\|}_{L^{p}(\mathbb{R}^{n})}\right\}.
    \end{equation*}
    Then $u_{m}(t):=f_{k(m,t)}$ defines a measurable function $u_{m}:E\rightarrow L^{p}(\mathbb{R}^{n})$ with finite image, satisfying ${\|u_{m}(t)\|}_{L^{p}(\mathbb{R}^{n})}\leq{\|u(t)\|}_{L^{p}(\mathbb{R}^{n})}$ and ${\|u_{m}(t)-u(t)\|}_{L^{p}(\mathbb{R}^{n})}\overset{m\rightarrow\infty}{\rightarrow}0$ for all $t\in E$.
\end{proof}
\subsection{The product operator}
The following lemma addresses measurability issues arising from the pointwise product $u\otimes v$ and the Bony product $\mathsf{Bony}(u,v)$.
\begin{lemma}\label{product-lemma}
\begin{enumerate}[label=(\roman*)]
    \item
    If $p\in[2,\infty]$, and $u,v:(0,T)\rightarrow\mathcal{S}'(\mathbb{R}^{n})$ are weakly* measurable functions satisfying $u(t),v(t)\in L^{p}(\mathbb{R}^{n})$ for almost every $t\in(0,T)$, then the pointwise product $u\otimes v$ defines an equivalence class of weakly* measurable functions $(0,T)\rightarrow\mathcal{S}'(\mathbb{R}^{n};\mathbb{R}^{n\times n})$.
    \item
    If $u,v\in\widetilde{L}_{T}^{\alpha}\dot{B}_{\ell,\infty}^{\frac{n}{\ell}-\frac{n}{p}}(\mathbb{R}^{n};\mathbb{R}^{n})$ for some $\alpha\in[2,\infty]$, $p\in(2,\infty)$ and $\ell\in[1,p)$, then the Bony product
    \begin{equation*}
        \mathsf{Bony}(u_{i},v_{j}) := \dot{T}_{u_{i}}v_{j} + \dot{T}_{v_{j}}u_{i} + \dot{R}(u_{i},v_{j})
    \end{equation*}
    converges in $\mathcal{S}'(\mathbb{R}^{n};\mathbb{R}^{n\times n})$ for almost every $t\in(0,T)$. The resulting product $\mathsf{Bony}(u,v)$ defines an equivalence class of weakly* measurable functions $(0,T)\rightarrow\mathcal{S}_{h}'(\mathbb{R}^{n};\mathbb{R}^{n\times n})$.
\end{enumerate}
\end{lemma}
\begin{proof}
\begin{enumerate}[label=(\roman*)]
    \item
    We may modify $u$ and $v$ on a zero-measure subset of $(0,T)$, such that $u(t),v(t)\in L^{p}(\mathbb{R}^{n})$ for all $t\in(0,T)$. We prove weak* measurability of $u\otimes v$ by writing $\langle u_{i}v_{j},\phi\rangle=\langle u_{i},v_{j}\phi\rangle$ for $\phi\in\mathcal{D}(\mathbb{R}^{n})$, and approximating the weakly* measurable function $v_{j}\phi:(0,T)\rightarrow L^{p'}(\mathbb{R}^{n})$ using Lemma \ref{pettis} (by Lemma \ref{measurability-duality}, the resulting functions $\langle u_{i},{(v_{j}\phi)}^{m}\rangle$ of $t$ are measurable, so in the limit the function $\langle u_{i},v_{j}\phi\rangle$ of $t$ is measurable).
    \item
    We have that $(\frac{n}{\ell}-\frac{n}{p},\ell,\infty)$ satisfy the infrared convergence condition \eqref{negative-scaling}, so we may modify $u$ and $v$ on a zero-measure subset of $(0,T)$, such that $\dot{\Delta}_{j}u(t),\dot{S}_{j}u(t),\dot{\Delta}_{j}v(t),\dot{S}_{j}v(t)\in L^{\infty}(\mathbb{R}^{n})$ for all $t\in(0,T)$. (The ability to do this for $\dot{S}_{j}u$ and $\dot{S}_{j}v$ follows from infrared convergence.). Then the Bony product is a sum of terms of the form $\widetilde{u}\otimes\widetilde{v}$, where $\widetilde{u}$ and $\widetilde{v}$ are weakly* measurable functions $(0,T)\rightarrow L^{\infty}(\mathbb{R}^{n})$, so weak* measurability of $\widetilde{u}\otimes\widetilde{v}$ can be proved as in part (i). By the calculation \eqref{B-product-estimate-2}, and the fact that $(\frac{n}{\ell}-\frac{2n}{p},\ell,\infty)$ satisfy \eqref{negative-scaling}, we have that the Bony product converges to an equivalence class of weakly* measurable functions $(0,T)\rightarrow\mathcal{S}_{h}'(\mathbb{R}^{n};\mathbb{R}^{n\times n})$.
\end{enumerate}
\end{proof}


\subsection{The operator $\mathbb{P}\nabla\cdot$}\label{sigma}
The following lemma describes how the operator $\mathbb{P}\nabla\cdot$ is defined on Lebesgue spaces, and how the operator ${(\mathbb{P}\nabla\cdot)}_{\varphi}:=\sum_{j\in\mathbb{Z}}\mathbb{P}\nabla\cdot\dot{\Delta}_{j}$ is defined on Besov spaces of sufficiently low regularity.
\begin{lemma}\label{sigma-lemma}
\begin{enumerate}[label=(\roman*)]
    \item
    For all $p\in[1,\infty]$, the expression
    \begin{equation*}
        \langle (\mathbb{P}\nabla\cdot W)_i,\phi\rangle := -\langle W_{jk},\mathbb{P}_{ij}\nabla_{k}\phi\rangle \qquad \text{for }\phi\in\mathcal{S}(\mathbb{R}^{n})
    \end{equation*}
    defines an operator $\mathbb{P}\nabla\cdot:L^{p}(\mathbb{R}^{n};\mathbb{R}^{n\times n})\rightarrow\mathcal{S}'(\mathbb{R}^{n};\mathbb{R}^{n})$ which commutes with $\dot{\Delta}_{j}$, and satisfies
    \begin{equation*}
        \mathbb{P}\nabla\cdot W = \sum_{j\in\mathbb{Z}}\mathbb{P}\nabla\cdot\dot{\Delta}_{j}W \quad \text{in }\mathcal{S}'(\mathbb{R}^{n}) \qquad \text{for all }W\in L^{p}(\mathbb{R}^{n};\mathbb{R}^{n\times n}).
    \end{equation*}
    \item
    If $W\in\dot{B}_{p,q}^{s}(\mathbb{R}^{n};\mathbb{R}^{n\times n})$ with $(s-1,p,q)$ satisfying \eqref{negative-scaling}, then the series
    \begin{equation*}
        {(\mathbb{P}\nabla\cdot)}_{\varphi} W := \sum_{j\in\mathbb{Z}}\mathbb{P}\nabla\cdot\dot{\Delta}_{j}W
    \end{equation*}
    converges in $\mathcal{S}'(\mathbb{R}^{n};\mathbb{R}^{n})$.
\end{enumerate}
\end{lemma}
\begin{proof}
    Let $\sigma$ be any smooth function on $\mathbb{R}^{n}\setminus\{0\}$ which is positive homogeneous of degree 1.

    The expression $\sigma(-D)=\mathcal{F}^{-1}\sigma(-\xi)\mathcal{F}=\mathcal{F}\sigma(\xi)\mathcal{F}^{-1}$ defines an operator $\sigma(-D):\mathcal{S}(\mathbb{R}^{n})\rightarrow L^{\infty}(\mathbb{R}^{n})$ which commutes with $\dot{\Delta}_{j}$. Applying dominated convergence to the inverse Fourier integral defining $\mathcal{F}^{-1}\sigma(-\xi)\mathcal{F}\phi$, we have
    \begin{equation}\label{sigma-series}
        \sigma(-D)\phi = \sum_{j\in\mathbb{Z}}\sigma(-D)\dot{\Delta}_{j}\phi \quad \text{almost everywhere} \qquad \text{for all }\phi\in\mathcal{S}(\mathbb{R}^{n}).
    \end{equation}
    By Lemma \ref{useful-inequalities}, for $p\in[1,\infty]$ we have
    \begin{equation}\label{sigma-series-estimate}
        \sum_{j\in\mathbb{Z}}{\|\sigma(-D)\dot{\Delta}_{j}\phi\|}_{L^{p'}(\mathbb{R}^{n})} \lesssim_{\sigma} {\|\phi\|}_{\dot{B}_{p',1}^{1}(\mathbb{R}^{n})} \qquad \text{for all }\phi\in\mathcal{S}(\mathbb{R}^{n}),
    \end{equation}
    where ${\|\phi\|}_{\dot{B}_{p',1}^{1}(\mathbb{R}^{n})}$ is controlled by finitely many seminorms in $\mathcal{S}(\mathbb{R}^{n})$ in view of Lemma \ref{convergence-duality}; in particular, it follows that  one not only has convergence almost everywhere in \eqref{sigma-series}, but also in $L^{p'}(\mathbb{R}^{n})$. By \eqref{sigma-series} and \eqref{sigma-series-estimate}, and the fact that $\sigma(-D)$ commutes with $\dot{\Delta}_{j}$, it follows that the expression
    \begin{equation*}
        \langle \sigma(D)w,\phi\rangle := \langle w,\sigma(-D)\phi\rangle \qquad \text{for }w\in L^{p}(\mathbb{R}^{n})\text{ and }\phi\in\mathcal{S}(\mathbb{R}^{n})
    \end{equation*}
    defines an operator $\sigma(D):L^{p}(\mathbb{R}^{n})\rightarrow\mathcal{S}'(\mathbb{R}^{n})$ which commutes with $\dot{\Delta}_{j}$, and satisfies
    \begin{equation*}
        \sigma(D)w = \sum_{j\in\mathbb{Z}}\sigma(D)\dot{\Delta}_{j}w \quad \text{in }\mathcal{S}'(\mathbb{R}^{n}) \qquad \text{for all }w\in L^{p}(\mathbb{R}^{n}).
    \end{equation*}

    On the other hand, by Lemma \ref{useful-inequalities} and Lemma \ref{convergence-lemma}, we have that if $w\in\dot{B}_{p,q}^{s}(\mathbb{R}^{n})$, with $(s-1,p,q)$ satisfying \eqref{negative-scaling}, then the series
    \begin{equation*}
        {\sigma(D)}_{\varphi}w := \sum_{j\in\mathbb{Z}}\sigma(D)\dot{\Delta}_{j}w
    \end{equation*}
    converges in $\mathcal{S}'(\mathbb{R}^{n})$; the use of Lemma \ref{convergence-lemma} is justified by noting that $\dot{\Delta}_{j}w$ is spectrally supported on frequencies $|\xi|\approx2^{j}$, so $\sigma(D)\dot{\Delta}_{j}w$ is spectrally supported on frequencies $|\xi|\approx2^{j}$.

    To deduce Lemma \ref{sigma-lemma}, we take $\sigma(\xi)=\mathrm{i}\frac{\xi\otimes\xi\otimes\xi}{{|\xi|}^{2}}$, $-\mathbb{P}_{ij}\nabla_{k}\phi=\sigma_{ijk}(-D)\phi$, and $\langle{(\mathbb{P}\nabla\cdot W)}_{i},\phi\rangle=\langle\sigma_{ijk}(D)W_{jk},\phi\rangle=\langle W_{jk},\sigma_{ijk}(-D)\phi\rangle$.
\end{proof}
When applying Lemma \ref{sigma-lemma} in the context of the space-time spaces $L_{T}^{\alpha}L^{p}(\mathbb{R}^{n})$ and $\widetilde{L}_{T}^{\alpha}\dot{B}_{p,q}^{s}(\mathbb{R}^{n})$, measurability issues are addressed as follows: with notation as in the proof of Lemma \ref{sigma-lemma}, if $\dot{\Delta}_{j}w:(0,T)\rightarrow L^{p}(\mathbb{R}^{n})$ is a weakly* measurable function, then for all $\phi\in\mathcal{S}(\mathbb{R}^{n})$ we have $\langle\sigma(D)\dot{\Delta}_{j}w,\phi\rangle = \langle\dot{\Delta}_{j}w,\sigma(-D)\phi\rangle$, where \eqref{sigma-series} and \eqref{sigma-series-estimate} tell us that $\sigma(-D)\phi\in L^{p'}(\mathbb{R}^{n})$; by Lemma \ref{measurability-duality}, it follows that $\sigma(D)\dot{\Delta}_{j}w:(0,T)\rightarrow \mathcal{S}'(\mathbb{R}^{n})$ is weakly* measurable.
\subsection{The operator $\langle Gw(t),\phi\rangle = \int_{0}^{t}\langle e^{(t-s)\Delta}w(s),\phi\rangle\,\mathrm{d}s$}
We give the
\begin{proof}[Proof of Lemma \ref{bilinear-G-lemma}]
    For $j\in\mathbb{Z}$ and almost every $s\in(0,T)$ we have
    \begin{equation}\label{G-measurability}
        [e^{t\Delta}\dot{\Delta}_{j}w(s)](x) = \langle\dot{\Delta}_{j}w(s),\widetilde{\Phi}(t,x)\rangle \quad \text{for all }(t,x)\in(0,\infty)\times\mathbb{R}^{n},
    \end{equation}
    where $[\widetilde{\Phi}(t,x)](y):=\Phi(t,x-y)$ defines a weakly* measurable function $\widetilde{\Phi}:(0,\infty)\times\mathbb{R}^{n}\rightarrow L^{1}(\mathbb{R}^{n})$, while $\dot{\Delta}_{j}w\in L_{T}^{\beta}L^{\infty}(\mathbb{R}^{n})$ by Lemma \ref{useful-inequalities}. Approximating $\widetilde{\Phi}$ using Lemma \ref{pettis}, we deduce that the expression \eqref{G-measurability} defines an equivalence class of measurable functions of  ${(s,t,x)\in(0,T)\times(0,\infty)\times\mathbb{R}^{n}}$. This addresses the measurability issues in the derivation of \eqref{Gj-estimate}, as well as in the following calculations.

    For every $t\in(0,T)$, and almost every $s\in(0,t)$, we have $w(s)\in\mathcal{S}_{h}'(\mathbb{R}^{n})$ and hence
    \begin{equation*}
        \langle e^{(t-s)\Delta}w(s),\phi\rangle = \sum_{j\in\mathbb{Z}}\langle e^{(t-s)\Delta}\dot{\Delta}_{j}w(s),\phi\rangle,
    \end{equation*}
    so by two applications of Fubini's theorem we have
    \begin{equation*}
    \begin{aligned}
        \langle Gw(t),\phi\rangle &= \int_{0}^{t}\sum_{j\in\mathbb{Z}}\langle e^{(t-s)\Delta}\dot{\Delta}_{j}w(s),\phi\rangle\,\mathrm{d}s \\
        &\overset{(1)}{=} \sum_{j\in\mathbb{Z}}\int_{0}^{t}\langle e^{(t-s)\Delta}\dot{\Delta}_{j}w(s),\phi\rangle\,\mathrm{d}s \\
        &\overset{(2)}{=} \sum_{j\in\mathbb{Z}}\int_{\mathbb{R}^{n}}\left(\int_{0}^{t}[e^{(t-s)\Delta}\dot{\Delta}_{j}w(s)](x)\,\mathrm{d}s\right)\phi(x)\,\mathrm{d}x.
    \end{aligned}
    \end{equation*}
    The use of Fubini's theorem in $\overset{(1)}{=}$ is justified by
    \begin{equation*}
    \begin{aligned}
        \int_{0}^{t}\sum_{j\in\mathbb{Z}}\left|\langle e^{(t-s)\Delta}\dot{\Delta}_{j}w(s),\phi\rangle\right|\,\mathrm{d}s &\leq \int_{0}^{t}\sum_{j\in\mathbb{Z}}\sum_{|\nu|\leq1}\left|\langle e^{(t-s)\Delta}\dot{\Delta}_{j}w(s),\dot{\Delta}_{j-\nu}\phi\rangle\right|\,\mathrm{d}s \\
        &\leq \int_{0}^{t}\sum_{j\in\mathbb{Z}}\sum_{|\nu|\leq1}{\|e^{(t-s)\Delta}\dot{\Delta}_{j}w(s)\|}_{L^{r}(\mathbb{R}^{n})}{\|\dot{\Delta}_{j-\nu}\phi\|}_{L^{r'}(\mathbb{R}^{n})}\,\mathrm{d}s \\
        &= \sum_{j\in\mathbb{Z}}\sum_{|\nu|\leq1}\widetilde{G}_{j,r}w(t){\|\dot{\Delta}_{j-\nu}\phi\|}_{L^{r'}(\mathbb{R}^{n})} \\
        &\lesssim_{n} T^{1-\frac{1}{\beta}}{\|w\|}_{\widetilde{L}_{T}^{\beta}\dot{B}_{r,q}^{s}(\mathbb{R}^{n})}{\left\|j\mapsto\sum_{|\nu|\leq1}2^{-js}{\|\dot{\Delta}_{j-\nu}\phi\|}_{L^{r'}(\mathbb{R}^{n})}\right\|}_{l^{q'}(\mathbb{Z})} \\
        &\lesssim_{n,r,s} T^{1-\frac{1}{\beta}}{\|w\|}_{\widetilde{L}_{T}^{\beta}\dot{B}_{r,q}^{s}(\mathbb{R}^{n})}{\|\phi\|}_{\dot{B}_{r',q'}^{-s}(\mathbb{R}^{n})},
    \end{aligned}
    \end{equation*}
    where we used \eqref{Gj-estimate}${}_{\gamma=\infty}$ in the fourth line. By Lemma \ref{convergence-duality}, the assumption that $(s,r,q)$ satisfy \eqref{negative-scaling} ensures that ${\|\phi\|}_{\dot{B}_{r',q'}^{-s}(\mathbb{R}^{n})}$ is controlled by finitely many seminorms in $\mathcal{S}(\mathbb{R}^{n})$. The use of Fubini's theorem in $\overset{(2)}{=}$ is justified by
    \begin{equation*}
    \begin{aligned}
        \int_{0}^{t}\left(\int_{\mathbb{R}^{n}}\left|[e^{(t-s)\Delta}\dot{\Delta}_{j}w(s)](x)\phi(x)\right|\,\mathrm{d}x\right)\mathrm{d}s &\leq \int_{0}^{t}{\|e^{(t-s)\Delta}\dot{\Delta}_{j}w(s)\|}_{L^{r}(\mathbb{R}^{n})}{\|\phi\|}_{L^{r'}(\mathbb{R}^{n})}\,\mathrm{d}s \\
        &\lesssim_{n} T^{1-\frac{1}{\beta}}2^{-js}{\|w\|}_{\widetilde{L}_{T}^{\beta}\dot{B}_{r,\infty}^{s}(\mathbb{R}^{n})}{\|\phi\|}_{L^{r'}(\mathbb{R}^{n})},
    \end{aligned}
    \end{equation*}
    where we used \eqref{Gj-estimate}${}_{\gamma=\infty}$ in the second line. We have therefore shown that $Gw(t)=\sum_{j\in\mathbb{Z}}G_{j}w(t)$ in $\mathcal{S}'(\mathbb{R}^{n})$ for all $t\in(0,T)$, where
    \begin{equation*}
        \langle G_{j}w(t),\phi\rangle = \int_{\mathbb{R}^{n}}\left(\int_{0}^{t}[e^{(t-s)\Delta}\dot{\Delta}_{j}w(s)](x)\,\mathrm{d}s\right)\phi(x)\,\mathrm{d}x = \int_{0}^{t}\langle e^{(t-s)\Delta}\dot{\Delta}_{j}w(s),\phi\rangle\,\mathrm{d}s.
    \end{equation*}
    The middle expression in this last equality consists of integrals of measurable functions, so the functions $G_{j}w:(0,T)\rightarrow\mathcal{S}'(\mathbb{R}^{n})$ are weakly* measurable, so $Gw$ is weakly* measurable. We establish \eqref{bilinear-G-projection} by writing
    \begin{equation*}
        \langle \dot{\Delta}_{j}Gw(t),\phi\rangle = \langle Gw(t),\dot{\Delta}_{j}\phi\rangle = \int_{0}^{t}\langle e^{(t-s)\Delta}w(s),\dot{\Delta}_{j}\phi\rangle\,\mathrm{d}s = \langle G_{j}w(t),\phi\rangle.
    \end{equation*}
\end{proof}
\subsection{The fixed point problem}
We give the
\begin{proof}[Proof of Proposition \ref{lorentz-besov-equiv}]
    \begin{enumerate}[label=(\roman*)]
        \item By the estimate \eqref{besov-heat-estimate-1} and the interpolation \eqref{interpolation-geometric}, we have $S[f]\in\widetilde{L}_{T}^{2}\dot{B}_{p_{0},\infty}^{s_{0}+1}(\mathbb{R}^{n})\cap\widetilde{L}_{T}^{2}\dot{B}_{p_{0},\infty}^{s_{0}}(\mathbb{R}^{n})\subseteq\widetilde{L}_{T}^{2}\dot{B}_{p_{0},1}^{0}(\mathbb{R}^{n})$. The estimate \eqref{besov-uniqueness-bilinear-estimate} tells us that if $u\in X_{\ell,\epsilon}^{\alpha}(T)$ then $B[u,u]\in\widetilde{L}_{T}^{\alpha}\dot{B}_{\ell,1}^{s_{\ell}+\epsilon+\frac{2}{\alpha}}(\mathbb{R}^{n})\subseteq\widetilde{L}_{T}^{\alpha}\dot{B}_{p,1}^{0}(\mathbb{R}^{n})$. For any $u,v\in\widetilde{L}_{T}^{2}\dot{B}_{p_{0},1}^{0}(\mathbb{R}^{n})+\widetilde{L}_{T}^{\alpha}\dot{B}_{p,1}^{0}(\mathbb{R}^{n})$, the Bony decomposition $u\otimes v=\mathsf{Bony}(u,v)$ is justified by the remarks preceding Lemma \ref{paraproduct}, and the identity $\mathbb{P}\nabla\cdot(u\otimes v)={(\mathbb{P}\nabla\cdot)}_{\varphi}(u\otimes v)$ is justified by the fact (from Lemma \ref{sigma-lemma}) that $\mathbb{P}\nabla\cdot$ agrees with ${(\mathbb{P}\nabla\cdot)}_{\varphi}$ when acting on Lebesgue spaces.
        \item Let $W\in L_{T}^{1}L^{p}(\mathbb{R}^{n};\mathbb{R}^{n\times n})$. For almost every $s\in(0,T)$ we have
        \begin{equation}\label{besov-lorentz-equiv-meas}
            \int_{\mathbb{R}^{n}}\nabla_{k}\mathcal{T}_{ij}(t,x-y)[W_{jk}(s)](y)\,\mathrm{d}y = \langle W_{jk}(s),\widetilde{\mathcal{T}}_{ijk}(t,x)\rangle \quad \text{for all }(t,x)\in(0,\infty)\times\mathbb{R}^{n},
        \end{equation}
        where $[\widetilde{\mathcal{T}}_{ijk}(t,x)](y):=\nabla_{k}\mathcal{T}_{ij}(t,x-y)$ defines a weakly* measurable function $\widetilde{\mathcal{T}}_{ijk}:(0,\infty)\times\mathbb{R}^{n}\rightarrow L^{p'}(\mathbb{R}^{n})$. (The estimate ${\|\nabla\mathcal{T}(t)\|}_{L^{q}(\mathbb{R}^{n})}\lesssim_{n}t^{-\frac{1}{2}\left(1+\frac{n}{q'}\right)}$ for $q\in[1,\infty]$ is well known; a proof can be found in \cite[Lemma 4.2]{davies-lorentz-blowup} or \cite[Proposition 11.1]{lemarie2002}.). Approximating either $W$ or $\widetilde{\mathcal{T}}_{ijk}$ using Lemma \ref{pettis} (depending on which of $p,p'$ is finite), we deduce that the expression \eqref{besov-lorentz-equiv-meas} defines an equivalence class of measurable functions of $(s,t,x)\in(0,T)\times(0,\infty)\times\mathbb{R}^{n}$. This addresses the measurability issues in the following calculation. For $\phi\in\mathcal{S}(\mathbb{R}^{n})$ we compute
        \begin{equation*}
        \begin{aligned}
            \langle G{(\mathbb{P}\nabla\cdot W)}_{i}(t),\phi\rangle &= \int_{0}^{t}\langle e^{(t-s)\Delta}{(\mathbb{P}\nabla\cdot W(s))}_{i},\phi\rangle\,\mathrm{d}s \\
            &= -\int_{0}^{t}\langle W_{jk}(s),\mathbb{P}_{ij}\nabla_{k}(e^{(t-s)\Delta}\phi)\rangle\,\mathrm{d}s \\
            &= -\int_{0}^{t}\left(\int_{\mathbb{R}^{n}}[W_{jk}(s)](y)\left(\int_{\mathbb{R}^{n}}\nabla_{k}\mathcal{T}_{ij}(t-s,y-x)\phi(x)\,\mathrm{d}x\right)\,\mathrm{d}y\right)\,\mathrm{d}s \\
            &= \int_{0}^{t}\left(\int_{\mathbb{R}^{n}}\left(\int_{\mathbb{R}^{n}}\nabla_{k}\mathcal{T}_{ij}(t-s,x-y)[W_{jk}(s)](y)\,\mathrm{d}y\right)\phi(x)\,\mathrm{d}x\right)\,\mathrm{d}s \\
            &= \int_{\mathbb{R}^{n}}\left(\int_{0}^{t}\left(\int_{\mathbb{R}^{n}}\nabla_{k}\mathcal{T}_{ij}(t-s,x-y)[W_{jk}(s)](y)\,\mathrm{d}y\right)\mathrm{d}s\right)\phi(x)\,\mathrm{d}x,
        \end{aligned}
        \end{equation*}
        where we used the identity $\nabla\mathcal{T}(t,-x)=-\nabla\mathcal{T}(t,x)$ in the fourth line. The use of Fubini's theorem in the above calculation, and the assertion $G\mathbb{P}\nabla\cdot W\in L_{T}^{1}L^{p}(\mathbb{R}^{n};\mathbb{R}^{n})$, are justified by the estimate
        \begin{equation*}
        \begin{aligned}
            &\quad \int_{0}^{T}\int_{0}^{t}\langle|\nabla\mathcal{T}(t-s)|*|W(s)|,|\phi|\rangle\,\mathrm{d}s\,\mathrm{d}t \\
            &\leq \int_{0}^{T}\int_{0}^{t}{\||\nabla\mathcal{T}(t-s)|*|W(s)|\|}_{L^{p}(\mathbb{R}^{n})}{\|\phi\|}_{L^{p'}(\mathbb{R}^{n})}\,\mathrm{d}s\,\mathrm{d}t \\
            &\leq \int_{0}^{T}\int_{0}^{t}{\|\nabla\mathcal{T}(t-s)\|}_{L^{1}(\mathbb{R}^{n})}{\|W(s)\|}_{L^{p}(\mathbb{R}^{n})}{\|\phi\|}_{L^{p'}(\mathbb{R}^{n})}\,\mathrm{d}s\,\mathrm{d}t \\
            &\lesssim_{n} \int_{0}^{T}\left(\int_{s}^{T}{(t-s)}^{-1/2}\,\mathrm{d}t\right){\|W(s)\|}_{L^{p}(\mathbb{R}^{n})}{\|\phi\|}_{L^{p'}(\mathbb{R}^{n})}\,\mathrm{d}s,
        \end{aligned}
        \end{equation*}
        and measurability of the map $s\mapsto{\|W(s)\|}_{L^{p}(\mathbb{R}^{n})}$ comes from Lemma \ref{measurable-norms}.
    \end{enumerate}
\end{proof}

\end{document}